\documentclass[10pt]{article}

\usepackage{ xcolor} 
\usepackage{amsmath, amsthm}
\usepackage{amssymb}
\usepackage{array}
\usepackage{amsfonts}
\usepackage{palatino}
\usepackage{graphicx}
\usepackage[normalem]{ulem}
\usepackage{tikz-cd}
\usepackage{hyperref}
\hypersetup{
    colorlinks = false,
    linkbordercolor = {blue}
}
\usepackage[capitalise]{cleveref}
\usepackage{mathdots}
\usepackage{framed}
\usepackage{enumitem}
\usepackage{caption}
\usepackage{lambda}

\crefname{equation}{Equation}{Equations}
\crefname{conjecture}{Conjecture}{Conjectures}
\crefname{figure}{Figure}{Figures}
\crefname{thm}{Theorem}{Theorems}

\newtheorem{thm}{Theorem}
\numberwithin{thm}{section} 
\newtheorem{lemma}[thm]{Lemma}
\newtheorem{cor}[thm]{Corollary}
\newtheorem{prop}[thm]{Proposition}

\theoremstyle{definition}

\newtheorem{example}[thm]{Example}

\theoremstyle{remark}
\newtheorem{rem}[thm]{Remark}

\newcommand{\Pet}{\ensuremath{\mathbf{P}}}
\newcommand{\Perm}{\ensuremath{\mathbf{Perm}}}

\newcommand{\Fl}{\ensuremath{G/B}}
\newcommand{\Hom}{\operatorname{Hom}}

\newcommand{\X}[1]{\ensuremath{\mathfrak X(#1)}}

\newcommand{\bo}{\ensuremath{\mathfrak b}}
\newcommand{\g}{\ensuremath{\mathfrak g}}
\newcommand{\car}{\ensuremath{\mathfrak h}}
\newcommand{\Hess}{\mathbf H}
\newcommand{\IC}{\mathbb C}
\newcommand{\IQ}{\mathbb Q}

\newcommand{\set}[2]{\left\{#1\,\middle\vert\,#2\right\}}

\DeclareMathOperator{\pr}{pr}

\author{Rebecca Goldin\footnote{rgoldin@gmu.edu}\ \footnote{Supported in part by NSF\#2152312}, Rahul Singh\footnote{rahul.sharpeye@gmail.com}}

\begin{document}
\title{Equivariant Chevalley, Giambelli, and Monk Formulae for the Peterson Variety}
\date{}
\maketitle

\begin{abstract}

We present a formula for the Poincar\'e dual in the flag manifold of the equivariant fundamental class of any regular nilpotent or regular semisimple Hessenberg variety as a polynomial in terms of certain Chern classes.
We then develop a type-independent proof of the Giambelli formula for the Peterson variety,
and use this formula to compute the intersection multiplicity of a Peterson variety with an opposite Schubert variety corresponding to a Coxeter word.
Finally, we develop an equivariant Chevalley formula for the cap product of a divisor class with a fundamental class,
and a dual Monk rule, for the Peterson variety.
\end{abstract}

\let\thefootnote\relax\footnote{AMS Subject Classes 14N10, 05E14}
\section{Introduction}
Let $G$ be a complex semisimple Lie group
corresponding to a Dynkin diagram $\Delta$. 
Let $B$ and $B^-$ be a pair of opposite Borel subgroups in $G$,
and let $T:= B \cap B^-$ be the corresponding maximal torus in $G$.
We identify $\Delta$ with the set of simple roots of $(G,B,T)$.
The quotient $G/B$ is the associated flag manifold, with a left $T$ action.

The $T$-equivariant homology $H_*^T(G/B)$ of $G/B$ has a basis consisting of {\em Schubert varieties,} which are $B$ orbit closures in $G/B$.
Since $G/B$ is smooth, there exists a dual basis of the $T$-equivariant cohomology $H_T^*(G/B)$
given by Schubert classes $\set{\sigma_w}{w\in W}$, where $W$ is the Weyl group.
This basis enjoys  {\em Graham positivity}, i.e., the structure constants $c_{uv}^w\in H_T^*(pt)$ defined by 
\begin{align*}
\sigma_u\sigma_v = \sum_{w\in W} c_{uv}^w \sigma_w
\end{align*}
are polynomials with non-negative coefficients in the set $\alpha\in \Delta$ for all $u,v,w\in W$.

This paper is concerned with a class of subvarieties of $G/B$  called {\em Hessenberg varieties}, with a particular focus  on the {\em Peterson variety}, a regular nilpotent Hessenberg variety.
Hessenberg varieties were first introduced by De Mari \cite{MR2636295} as part of the study of certain matrix decomposition algorithms.
They were generalized outside type A by De Mari, Procesi, and Shayman \cite{MR1043857},
who further identified the permutahedral variety, a toric variety studied by Klyachko \cite{klyachko85,klyachko95}, as a particular Hessenberg variety.
The Peterson variety is a flat degeneration of the permutahedral variety \cite{abe.fujita.zeng, MR2609167}.
Peterson \cite{peterson:notes} and Kostant \cite{kostant:flag} showed that the coordinate ring of a particular open affine subvariety of the Peterson variety is isomorphic to the quantum cohomology ring of the flag variety; see also \cite{rietsch:totally}.

The Peterson variety admits a natural $\IC^*$ action.
In \cite{gms:peterson}, Goldin, Mihalcea, and Singh show that $\IC^*$-equivariant Peterson Schubert calculus also satisfies Graham positivity; see \cref{eq:pstruct} below.
The equivariant cohomology of Peterson varieties in all Lie types is described in \cite{harada.horiguchi.masuda:Peterson}, using generators and relations.
Drellich \cite{drellich:monk} found a Giambelli formula for certain Coxeter elements (see \cref{intro:Giambelli}) using a type-by-type analysis.
Horiguchi \cite{horiguchi2021mixed} obtained a Monk rule for ordinary cohomology using similar methods.
In type A, an equivariant Monk rule was developed by Harada and Tymoczko \cite{harada.tymoczko:Monk}. Goldin and Gorbutt \cite{goldin.gorbutt:PetSchubCalc} subsequently found positive combinatorial formulae for all equivariant structure constants for the Peterson variety in type A. In parallel, Abe, Horiguchi, Kuwata and Zheng provided a different formula for the nonequivariant structure constants \cite{AbHoKuZe}.
 
In this paper, we prove an equivariant Giambelli formula for all Coxeter elements, in any Lie type, using type-independent methods.
We also prove an equivariant Monk rule (\cref{intro:thm:monk}) and a dual equivariant Chevalley formula (\cref{intro:thm:chevalley}), both of which use a pairing of equivariant cohomology and homology for the Peterson variety (see \cref{intro:thm:mult}). 
 
We denote by $\Phi$, $\Phi^+$, and $W$ the set of roots, set of positive roots, and Weyl group, respectively.
Let $\g=Lie(G)$, $\bo=Lie(B)$, $\car=Lie(T)$,
and let $\g_\alpha\subset\g$ denote the root space corresponding to $\alpha\in\Phi$.
Consider the subspace 
\begin{align*}
H_0=\bo\oplus\bigoplus_{\alpha\in\Delta}\g_{-\alpha}.
\end{align*}
A $B$-stable subspace $H$ of \g\ containing $H_0$ is called an \emph{indecomposable Hessenberg space}.
For $x\in\g$, 
we have a corresponding \emph{Hessenberg variety},
\begin{align*}
\Hess(x,H)=\set{gB\in G/B}{Ad(g^{-1})x\in H}.
\end{align*}
The Hessenberg variety $\Hess(x,H)$ admits an action by the projective centralizer of $x$,
given by $
\widetilde C_G(x)=\set{g\in G}{Ad(g)x=\lambda x\text{ for some }\lambda\in\IC}.
$

{For each $\alpha\in\Phi^+$, let $\alpha^\vee$ denote the corresponding coroot. For $\alpha\in \Delta$, fix a non-zero element $e_\alpha\in\g_\alpha$. Let}
\begin{align*}
h=\sum_{\alpha\in\Phi^+}\alpha^\vee\in\car,&&
e=\sum_{\alpha\in\Delta} e_\alpha.
\end{align*} 
The Hessenberg variety $\Pet:=\Hess(e,H_0)$ is called the \emph{Peterson variety},
and the Hessenberg variety $\Perm:=\Hess(h,H_0)$
is called the \emph{permutahedral variety}.

Recall that $h$ is a regular semisimple element, with $\widetilde C_G(h)=T$,
and hence $T$ acts on $\Hess(h,H)$.
Let $\mathfrak s=\IC h$ be the one-dimensional Lie algebra spanned by $h$,
and let $S\subset T$ be the one-dimensional torus corresponding to $\mathfrak s\subset\mathfrak h$.
Since $[h,e]=2e$, we have $S\subset \widetilde C_G(e)$,
and hence an $S$-action on $\Hess(e,H)$.
Let $t$ be the image of a simple root $\alpha$ under the natural restriction $\car^*\to\mathfrak s^*$.
The element $t$ is independent of the choice of $\alpha$,
and further, $H^*_S(pt;\IQ)=\IQ[t]$.

Under the mild assumption $H_0\subset H$,
Abe, Fujita, and Zeng \cite[Cor. 3.9]{abe.fujita.zeng} have identified the Poincar\'e dual of the fundamental class $[\Hess(x,H)]$ in ordinary cohomology
as the Euler class of the vector bundle $G\times^B(\g/H)$.
Our first result (\cref{HessClass}) is an extension of their result to the equivariant setup.
In type A, this result was first proved by Anderson and Tymozcko in \cite{MR2609167}.

For $\lambda$ a character of $T$,
we denote by $\mathcal L_\lambda\to G/B$ the line bundle $\mathcal L_\lambda:=G\times^B\IC_{-\lambda}$.
Let $c_1^T$ (resp. $c_1^S$) denote the $T$ (resp. $S$)-equivariant first Chern class.

\begin{thm}
\label{intro:HessClass}
Suppose $H_0\subset H$.
We have the following equalities in the equivariant homology of the flag manifold $G/B$: 
\begin{align*}
[\Hess(e,H)]_S&=\prod\limits\left(c_1^S(\mathcal L_\alpha)-t\right)\cap[\Fl]_S,\\
[\Hess(h,H)]_T&=\prod\left(c_1^T(\mathcal L_\alpha)\right)\cap[\Fl]_T,
\end{align*}
where the product is over the set $\set{\alpha\in\Phi}{\g_{-\alpha}\not\subset H}$. 
\end{thm}

Recall the Schubert classes $\sigma^S_w\in H^{\ell(w)}_S(G/B)$,
Poincar\'e dual to the Schubert varieties $X^w=\overline{B^-wB/B}$.
Let $i^*:H_S^*(G/B)\to H_S^*(\Pet)$ be the pullback induced by the inclusion $i:\Pet\to G/B$,
and let $p_v=i^*\sigma^S_v$.
For convenience, we write $\sigma_\alpha^S=\sigma_{s_\alpha}^S$ and $p_\alpha=i^*\sigma^S_\alpha$ for $\alpha\in\Delta$.
An element $v_I\in W$ is called a \emph{Coxeter element} for some $I\subset\Delta$
if each simple reflection $s_\alpha$, $\alpha\in I$ appears exactly once in a reduced expression of $v_I$.

Fix a Coxeter element $v_I$ for each $I\subset\Delta$.
In \cite{gms:peterson} we prove that $\set{p_{v_I}}{I\subset\Delta}$ is a basis for $H^*_S(\Pet)$ in all Lie types, and that the structure constants $c_{IJ}^K\in H_S^*(pt)$ defined by the equation
\begin{equation}\label{eq:pstruct}
p_{v_I}p_{v_J} = \sum_{K\subset \Delta} c_{IJ}^K p_{v_K}
\end{equation}
are polynomials in $t$ with non-negative coefficients.
The positivity of $c_{IJ}^K$ can also be derived from \cite{drellich:monk}.

Our next result (\cref{Giambelli})
is an equivariant Giambelli formula, 
expressing the pullback of a \emph{Coxeter Schubert class} as a polynomial in the divisor classes $p_\alpha$.
\begin{thm}[Giambelli formula]
\label{intro:Giambelli}
Let $v_I$ be a Coxeter element for $I$, 
let $R(v_I)$ be the number of reduced words for $v_I$,
and let $\Omega_I=\prod_{\alpha\in I}p_\alpha\in H^*_S(\Pet)$.
We have
\begin{align}
\label{intro:eq:Giambelli}
p_{v_I}=\frac{R(v_I)}{|I|!} \Omega_I.
\end{align}
\end{thm}

\Cref{intro:eq:Giambelli} was first obtained by Drellich \cite{drellich:monk}
for a particular choice of Coxeter element $v_I$ for each $I$,
using a type by type analysis,
and the localization formula of Andersen, Jantzen, and Soergel \cite{AJS},
and Billey \cite{billey:kostant}.
Our proof has the benefit of being type independent, and working for all Coxeter elements.

We sketch an alternate proof of \cref{intro:Giambelli}, suggested by a referee.
Drellich's proof \cite[Lemma 5.1]{drellich:monk}, stating $p_{v_I}=C \Omega_I$ for some $C\neq 0$,
does not depend on the choice of Coxeter element made by the author.
On the other hand it follows from \cite[Thm. 4]{klyachko85} that $p_{v_I}=\frac{R(v_I)}{|I|!} \Omega_I$ in the ordinary cohomology $H^*(\Perm)$ for any Coxter element $v_I$.
These observations, combined with the isomorphism $H^*(\Pet)\cong H^*(\Perm)^W$ proved in \cite{MR4116638}, results in \cref{intro:Giambelli}.

The Peterson variety admits a cell-stratification,
$
\Pet=\bigsqcup_{I\subset\Delta}\Pet_I^\circ,
$
see \cite{tymoczko:paving,precup,balibanu:peterson}.
Consequently, we have a natural basis for the equivariant homology $H_*^S(\Pet)$,
given by the fundamental classes $\set{[\Pet_I]_S}{I\subset\Delta}$.
Let 
\begin{align*}
\left\langle\ ,\,\right\rangle:H^*_S(\Pet)\times H_*^S(\Pet)\to H^*_S(pt)
\end{align*}
 be the pairing given by equivariant integration,
i.e., $\left\langle\omega,[Z]\right\rangle=\int_{[Z]}\omega$. 
Following \cite[Thm 1.1]{gms:peterson},
if $v_I$ is a Coxeter element for some $I\subset\Delta$,
then
\begin{align}
\label{intro:dual}
\left\langle p_{v_I},[\Pet_J]_S\right\rangle=m(v_I)\delta_{IJ}
\end{align}
for some positive integer $m(v_I)$.
In particular, $\set{p_{v_I}}{I\subset\Delta}$ is a basis of $H_S^*(\Pet)$,
dual (up to scaling) to the fundamental class basis $\set{[\Pet_I]}{I\subset\Delta}$.

The multiplicities $m(v_I)$ depend only on $I$ and not on $\Delta$.
An algorithm for computing these multiplicities using the localization of a Schubert class to a specific point was provided in \cite[Prop 7.3]{gms:peterson},
resulting in a formula for the $m(v_I)$ for certain Coxeter elements $v_I$.
We prove in \cref{thm:mult} a general formula for $m(v_I)$, for any Coxeter element $v_I$, confirming the conjecture in \cite[Remark 7.7]{gms:peterson}.


Let $C_I$ be the Cartan matrix of $I$.
Recall that $f_I:=\det(C_I)$ is called the \emph{connection index} of $I$, see \cite{bourbaki:Lie46}.

\begin{thm}[Multiplicity Formula]
\label{intro:thm:mult}
Let $v_I$ be a Coxeter element of $I$,
and let $R(v_I)$ denote the number of reduced expressions for $v_I$.
We have 
\begin{equation*}
m(v_I)=\left\langle p_{v_I},[\Pet_I]_S\right\rangle=\frac{R(v_I)|W_I|}{|I|!\,f_I}.
\end{equation*}
\end{thm}

Let us say a few words about the proofs of \cref{intro:Giambelli,intro:thm:mult}.
The Hessenberg varieties $\Hess(x,H)$,
as $x$ varies over the set of regular elements in $\mathfrak g$,
form a flat family of subvarieties in $G/B$,
see \cite[Prop. 6.1]{abe.fujita.zeng}.
Abe, Horiguchi, Masuda, Murai, and Sato \cite{MR4116638} proved the following diagram is commutative.
\begin{equation}
\label{intro:commTri}
\begin{tikzcd}
&H^*(\Fl)\arrow[dl,swap,"i^*",twoheadrightarrow]\arrow[d,twoheadrightarrow,dashed]\arrow[dr,"j^*"]&\\
H^*(\Pet)\arrow[r,"\sim"]&H^*(\Perm)^W\arrow[hook,r]&H^*(\Perm).
\end{tikzcd}
\end{equation}
In type A, this was earlier proved in \cite{ahhm:cohomologyring}.
Here $j^*:H^*(G/B)\to H^*(\Perm)$ is the pullback induced by the inclusion $j:\Perm\to G/B$.
This allows us to relate computations in the ordinary cohomology of the Peterson variety
to corresponding computations in the ordinary cohomology of the permutahedral variety.
In \cite{klyachko85,klyachko95} (see also \cite{nadeau.tewari}),
Klyachko presented a Giambelli formula expressing the pullback class
$j^*\sigma_w$ of any Schubert class as a polynomial in the divisor classes $j^*\sigma_{s_\alpha}$,
\begin{align}
\label{intro:eq:klyachko}
j^*\sigma_v=\frac{1}{\ell(v)!}\sum_{\underline v\in\mathcal R(v)}\prod_{s_\alpha\in\underline v}j^*\sigma_{s_\alpha}.
\end{align}
Here $\mathcal R(v)$ is the set of reduced expressions for $v$,
and the product is over all occurrences of $s_\alpha$ in the reduced expression $\underline v$.
Following \cref{intro:commTri}, the same relation holds amongst the $i^*\sigma_v$ and $i^*\sigma_\alpha$ (in ordinary cohomology).
\Cref{intro:Giambelli} is an equivariant version of \cref{intro:eq:klyachko}.
We then use \cref{intro:Giambelli} and the duality (\cref{intro:dual}) 
to reduce the calculation of $m(v_I)$ in \cref{intro:thm:mult} to the non-equivariant integral 
$\int_\Perm\prod_{\alpha\in\Delta}\sigma_{s_\alpha}$,
which we found in \cite{klyachko85}.

In our final results, we develop a \emph{Chevalley formula} (\cref{thm:chevalley}),
and a dual \emph{Monk rule} (\cref{thm:monk}).
A Monk rule for the ordinary cohomology of \Pet\ was recently obtained by Horiguchi \cite{horiguchi2021mixed}.
The equivalence of the Chevalley formula and the Monk rule is a consequence of \cref{intro:dual,intro:thm:mult}.

Recall that for any Dynkin subdiagram $J\subset\Delta$,
we have unique elements $\varpi_\alpha^J$ in the weight lattice of $J$,
called the fundamental weights,
satisfying $\left\langle\beta^\vee,\varpi_\alpha^J\right\rangle=\delta_{\alpha\beta}$ for all $\beta\in J$.
Similarly, we have fundamental coweights $\varpi_\alpha^{J\vee}$ in the coweight lattice of $J$,
dual to the roots $\alpha\in J$.
We write $\varpi_\alpha$ (resp. $\varpi_\alpha^\vee$) for the fundamental weights (resp. coweights) for $\Delta$.
In general, we have
$\varpi_\alpha^I\neq\varpi_\alpha$ and 
$\varpi^{I\vee}_\alpha\neq\varpi^\vee_\alpha$, see \cref{sec:stab}.

\begin{thm}[Equivariant Chevalley formula]
\label{intro:thm:chevalley}
For $\alpha\in\Delta$, $J\subset\Delta$, we have
\begin{align*}
p_\alpha\cap[\Pet_J]_S=
\begin{cases}
0 &\text{if }\alpha\not\in J,\\
\left\langle2\rho_J^\vee,\varpi_\alpha\right\rangle t\,[\Pet_J]_S+
\sum\limits_{\substack{\beta\in J\\ K=J\backslash\{\beta\}}}
\left\langle\varpi^{J\vee}_\beta,\varpi^J_\alpha\right\rangle
\frac{|W_J|}{|W_K|}[\Pet_K]_S
&\text{if }\alpha\in J.
\end{cases}
\end{align*}
Here $\rho_J^\vee=\frac 12\sum_{\alpha\in\Phi_J^+}\alpha^\vee$ is one-half the sum of the positive coroots supported on $J$,
and $W_J$ and $W_K$ are the Weyl subgroups of the Dynkin diagrams $J$ and $K$ respectively.
\end{thm}

It is common in the literature 
(see, for example, \cite{drellich:monk,insko.tymoczko:intersection.theory,goldin.gorbutt:PetSchubCalc})
to fix a Coxeter element $v_I$ for each $I\subset\Delta$,
and to work with the basis $\set{p_{v_I}}{I\subset\Delta}$ of $H^*_S(\Pet)$. 
Following \cref{intro:Giambelli}, $\set{\Omega_I:=\prod_{\alpha\in I}p_\alpha}{I\subset\Delta}$ is also a basis of $H_S^*(\Pet)$,
related to the basis $\set{p_{v_I}}{I\subset\Delta}$ via a diagonal change of basis matrix.
We develop a Monk rule for the basis $\{\Omega_I\}$, resulting in a formula that does not depend a choice of Coxeter element for each $I$. For the reader's convenience, we present the Monk rule for the basis $\set{p_{v_I}}{I\subset\Delta}$ in \cref{rem:monk}.


\begin{thm}[Equivariant Monk Rule]
\label{intro:thm:monk}
For $\alpha\in\Delta$, we have
\begin{align*}
\Omega_\alpha \Omega_I=
\begin{cases}
\Omega_{I\cup\{\alpha\}}
&\text{if }\alpha\not\in I,\\
2\left\langle\rho^\vee_I,\varpi_\alpha\right\rangle t \Omega_I+
\sum\limits_{\substack{\gamma\in\Delta\backslash I\\J=I\cup\{\gamma\}}}
\dfrac{f_J}{f_I}\left\langle\varpi_\gamma^{J\vee},\varpi^J_\alpha\right\rangle \Omega_J
&\text{if }\alpha\in I.
\end{cases}
\end{align*}
\end{thm}

Consider $\alpha\in\Delta$, and let $I=\Delta\backslash\{\alpha\}$.
A key step in the proof of
 \cref{intro:thm:chevalley} is the following formula (see \cref{qChevalley}):
\begin{align}
\label{intro:qChevalley}
(c_1^S(i^*\mathcal L_\alpha)-t)\cap[\Pet]_S=\frac{|W|}{|W_I|}[\Pet_I]_S.
\end{align}

In  \cite{harada.horiguchi.masuda:Peterson}, Harada, Horiguchi, and Masuda present a description of $H^*_S(\Pet)$ by generators and relations. As a further consequence of \cref{qChevalley}, we
obtain a new proof of their description; 
see \cref{hhm}.

Let us now outline the organization of the paper.
In \cref{sec:cohomology}, we recall some results on the equivariant cohomology of spaces with affine paving,
as developed by Edidin, Graham, and Kreiman in \cite{edidin.graham,graham:positivity,graham2020cominuscule}.
In \cref{sec:preliminaries},
we recall some results on root systems, flag manifolds, and Schubert varieties.
In \cref{sec:Hessenberg}, we describe Hessenberg varieties,
and compute the Poincar\'e dual of the fundamental class of a regular Hessenberg variety as a polynomial in the Chern classes of line bundles
(\cref{intro:HessClass}).
We also recall from \cite{tymoczko:paving,precup,gms:peterson} some results on Peterson varieties.
In \cref{sec:giambelli}, we describe the relationship obtained in \cite{MR4116638}
between the ordinary cohomology rings
$H^*(\Hess(e,H))$, $H^*(\Hess(h,H))$, and $H^*(G/B)$,
see \cref{intro:commTri}.
We give a different explanation of this relationship,
following ideas developed by Brosnan and Chow \cite{brosnan-chow},
and further refined by B{\u{a}}libanu and Crooks \cite{balibanu2020perverse}.
We then use \cref{intro:commTri} to prove the Giambelli formula (\cref{intro:Giambelli})
and the multiplicity formula (\cref{intro:thm:mult}).
Finally, in \cref{sec:chevalley},
we prove the Chevalley formula (\cref{intro:thm:chevalley}),
its dual Monk rule (\cref{intro:thm:monk}),
and recover the Harada-Horiguchi-Masuda presentation of $H^*_S(\Pet)$.
We also tabulate the structure constants appearing in the Chevalley and Monk formulae,
and present some examples applying these fomulae.

\emph{Acknowledgements:}
We would like to thank both Ana B\u alibanu and Peter Crooks for explaining the results of \cite{balibanu2020perverse} to us, and 
Leonardo Mihalcea for several illuminating discussions. We also thank the referee for a careful and close reading of the manuscript, leading to improvements of several proofs.
Computer calculations in service of this paper were coded in SageMath \cite{sagemath}. 
Parts of this work were conducted while RS was at Virginia Tech,
and parts while at ICERM. RG was supported by National Science Foundation grant \#2152312.
We gratefully acknowledge the support of these institutions.

\section{Equivariant (Co)Homology}
\label{sec:cohomology}

Let $X$ be a complex algebraic variety equipped with a left action of a torus $T$.
We recall aspects of the $T$-equivariant homology and cohomology of $X$.
We will use the Borel model of equivariant cohomology,
and equivariant Borel-Moore homology,
following the setup in Graham's paper \cite{graham:positivity}. 
We refer to \cite[Ch~19]{fulton:IT}, \cite[Appendix~B]{fulton:young}, and \cite[\S 2.6]{chriss.ginzburg}
for more details about cohomology and Borel-Moore homology. 
\emph{We will study (co)homology with rational coefficients}.

%
%

Recall that the cohomology ring $H^*_T(pt)$ of a point is naturally identified with 
$\operatorname{Sym}(\car^*)$, 
the symmetric algebra of the dual of the Lie algebra of $T$. 
The morphism $X \to \{pt\}$ from $X$ to a point
gives the equivariant cohomology $H^*_T(X)$ the structure
of a graded algebra over $H^*_T(pt)$ via the pullback map $H_T^*(pt)\to H_T^*(X)$.
In addition, the cap product
\begin{align*}
\cap: H^k_T(X)\times H_\ell^T(X)\to H_{\ell-k}^T(X)
\end{align*}
endows the
equivariant homology $H_*^T(X)$ with a graded module structure over $H^*_T(X)$. 
Equivalently, there is a compatibility of cap and cup products given by
\begin{align*}
&&(a \cup b) \cap c = a \cap (b \cap c)
&&a,b \in H^*_T(X),\ c \in H_*^T(X).
\end{align*} 

For any map $S\to T$ of tori, 
we have a natural map of algebras $H^*_T(X) \to H^*_S(X)$,
compatible with the algebra map $H^*_T(pt)\to H^*_S(pt)$ induced by $Lie(T)^*\to Lie(S)^*$.
In particular, taking $S$ to be the trivial subgroup in $T$,
we obtain the restriction to ordinary cohomology,
$
H^*_T(X)\to H^*(X)
$.

\subsection{The integration pairing}
\label{sec:pairing}
Each irreducible, $T$-stable, closed subvariety $Z \subset X$ of complex dimension 
$k$ has a fundamental class $[Z]_T \in H_{2k}^T(X)$. If $X$ is smooth and irreducible,
then there exists a unique class $\eta_Z \in H^{2 (\dim X - k)}_T(X)$,
called the Poincar{\'e} dual of $Z$,
such that \begin{equation*}\eta_Z \cap [X]_T = [Z]_T.\end{equation*}

Given a $T$-equivariant proper map $f:X\to Y$, 
there is a push-forward $f_*:H_i^T(X) \to H_i^T(Y)$,
determined by the fact that if $Z \subset X$ is irreducible and $T$-stable, 
then
\begin{align*}
f_*([Z])=
\begin{cases}
d_Z[f(Z)]&\text{if }\dim f(Z)=\dim Z,\\
0&\text{if }\dim f(Z)<\dim Z,
\end{cases}
\end{align*}
where $d_Z$ is the generic degree of the restriction $f:Z \to f(Z)$.
The push-forward and pull-back are related by the projection formula 
\begin{equation}
\label{projectionFormula}
f_* (f^*(\eta) \cap c) = \eta \cap f_*(c),
\end{equation}
for $\eta\in H^*_T(Y)$ and $c\in H^T_*(X)$.
Recall that we have an isomorphism
\begin{align}
\label{pointPair}
&&H^j_T(pt)\xrightarrow\sim H_{-j}^T(pt),&&
a\mapsto a\cap[pt]_T.
\end{align} 
In particular, $H^*_T(pt)$ lives in non-negative degrees,
and $H_*^T(pt)$ lives in non-positive degrees.

Suppose now that $X$ is complete, so that $f: X \to pt$ is proper. 
For a homology class $c \in H_{-j}^T(X)$, we denote by 
$\int_X c$ the class $f_*(c) \in H_{-j}^T(pt)$,
viewed as an element of $H_T^{j}(pt)$ via  \cref{pointPair}.
Then we may define a pairing,
 \begin{equation}
\label{intPairing}
\langle\ , \,\rangle : H^i_T(X) \mathop
\times
 H_{-j}^T(X) \to H^{i+j}_T(pt) \/; \quad 
\langle \eta, c \rangle := \int_X \eta \cap c \/. \end{equation}
The pairing in \cref{intPairing} is compatible with the pairing in ordinary (co)homology.
We have forgetful maps $H^*_T(X)\to H^*(X)$ and $H_*^T(X)\to H_*(X)$,
and a commutative diagram,
\begin{center}
\begin{tikzcd}
H^i_T(X) \mathop\times H_{-j}^T(X)\arrow[r,"\langle\ \text{,}\ \rangle"]\arrow[d]&H^{i+j}_T(pt)\arrow[d]\\
H^i(X) \mathop\times H_{-j}(X)\arrow[r,"\langle\ \text{,}\ \rangle"]& H^{i+j}(pt).
\end{tikzcd}
\end{center}

\subsection{Spaces with affine paving}
Following \cite[Ex~1.9.1]{fulton:IT} (see also \cite{graham:positivity}) we say that 
a $T$-variety $X$ admits a {\em $T$-stable affine paving}
if it admits a filtration $X:=X_n \supset X_{n-1} \supset \ldots$ by closed $T$-stable subvarieties
such that each $X_i \setminus X_{i-1}$ is a finite disjoint union
of $T$-invariant varieties $U_{ij}$ isomorphic to affine spaces $\mathbb A^i$. 

\begin{lemma}[cf. \cite{graham:positivity}] 
\label{lemma:generate} Assume $X$ admits a $T$-stable affine paving, with cells $U_{ij}$. 
\begin{enumerate}[label=(\alph*)]
\item
The equivariant homology $H_*^T(X)$ is a free $H^*_T(pt)$-module with 
basis $\{[\overline{U_{ij}}]_T\}$.
\item
If $X$ is complete, the pairing from \cref{intPairing} is perfect, 
and so we may identify
$H^*_T(X) = \Hom_{H^*_T(pt)}(H_*^T(X), H^*_T(pt))$.
\end{enumerate}
\end{lemma}

\subsection{Chern classes and Euler classes}
\label{sec:eulerClass}
We will denote by $c_i^T(\ )$ the $i^{th}$ $T$-equivariant Chern class,
and by $e^T(\ )$ the equivariant Euler class of a $T$-equivariant vector bundle.
We say that a section $s$ of a vector bundle $\mathcal V\to X$ is regular if the codimension of the zero set of $s$ equals the rank of $\mathcal V$.
Recall that the torus $T$ acts on the space $H^0(X,\mathcal V)$ of sections of $\mathcal V$ via the formula $(z\cdot s)(x) = zs(z^{-1}x)$ for all $x\in X$ and $z\in T$.
The sections which are invariant under this action are precisely those that intertwine the action,
i.e., satisfy $s(zx)=zs(x)$ for all $x\in X$ and $z\in T$.

\begin{lemma}{\rm (cf. \cite[Lemma 2.2]{graham2020cominuscule})}
\label{lem:eqSec}
If $\mathcal L$ is a $T$-equivariant line bundle on a $T$-scheme $X$,
and $s$ is a $T$-invariant regular section of $\mathcal L$ with zero-scheme $Y$,
then 
\begin{equation*}
[Y]_T=c_1^T(\mathcal L)\cap[X]_T.
\end{equation*}
\end{lemma}

\begin{cor}
\label{cor:EulerClass}
If $\mathcal V$ is a $T$-equivariant vector bundle on a $T$-scheme $X$,
and $s$ is a $T$-invariant regular section of $\mathcal V$ with zero-scheme $Y$,
then $[Y]_T=e^T(\mathcal V)\cap[X]_T$.
\end{cor}

For $\lambda$ a character of $T$, let 
$\underline{\IC}_{\lambda} = X\times\IC\to X$
denote the (geometrically trivial) equivariant line bundle,
with $T$-action given by $z(x,v)=(zx,\lambda(z)v)$ for all $z\in T$.
By a standard abuse of notation, we write $\lambda$ for the $T$-equivariant first Chern class of $\underline\IC_{\lambda}$.

\begin{cor}
\label{cor:twistedSection} 
Let $\mathcal V\to X$ be  a $T$-equivariant vector bundle. 
For a character $\lambda$ of $T$, 
let $s$ be a regular section of $\mathcal V$ that lies in the $\lambda$-weight space, i.e.,
$(z\cdot s)=\lambda(z) s $ for all $z\in T$.
The zero scheme $Z(s)$ of $s$ is $T$-invariant, and we have
\begin{equation*}
[Z(s)]_T=e^{T}\left(\mathcal V\otimes\underline\IC_{-\lambda}\right)\cap[X]_T.
\end{equation*}
If $\mathcal V$ admits a filtration with $T$-equivariant line bundle quotients $\{\mathcal L_i\}$,
we have
\begin{equation*}
[Z(s)]_T=\prod\left(c_{1}^{T}(\mathcal L_i)-\lambda\right)\cap[X]_T.
\end{equation*}
\end{cor}
\begin{proof}
If $\lambda$ is non-trivial, the section $s$ is not invariant.
Observe however that the section $r=s\otimes 1$
of the vector bundle $\mathcal V\otimes\underline\IC_{-\lambda}$ is $T$-invariant, since
\begin{align*}
r(z x)&=s(z x)\otimes 1=\left(zz^{-1}s(zx)\right)\otimes 1\\
&= \left(z\left( \left(z^{-1}\cdot s\right)(x) \right)\right) \otimes 1 \\
&= \left(z\left( \lambda(z^{-1})s(x) \right)\right)\otimes 1\\
& = \left(\lambda(z^{-1})zs(x)\right)\otimes 1 \\
&=zs(x)\otimes \lambda(z^{-1})=z(s(x)\otimes 1)=zr(x).
\end{align*}
Further, $r$ has the same zero scheme as $s$, i.e., $Z(r)=Z(s)$.
Hence the first equality follows from \cref{cor:EulerClass}.

Suppose $\mathcal V$ admits a filtration by line bundles $\{\mathcal L_i\}$.
Then $\mathcal V\otimes\underline\IC_{-\lambda}$
admits a filtration by line bundles $\{\mathcal L_i\otimes\underline\IC_{-\lambda}\}$.
Applying the Whitney splitting principle, we have
\begin{align*}
e^T\left(\mathcal V\otimes\underline\IC_{-\lambda}\right)
=\prod\left(c_1^{T}\left(\mathcal L_i\otimes\underline\IC_{-\lambda}\right)\right)
=\prod\left(c_1^{T}(\mathcal L_i)-\lambda\right),
\end{align*}
from which the second equality follows.
\end{proof}


\section{Flag Manifolds}
\label{sec:preliminaries}

Fix a complex semisimple Lie group $G$,
opposite Borel subgroups $B, B^- \subset G$,
and let $T= B \cap B^-$ be the common maximal torus.
We will further assume that $G$ is simply connected;
this ensures that all line bundles on the flag manifold $G/B$ are $T$-equivariant.
Denote by $\Delta$ the system of simple positive roots associated to $(G,B,T)$,
by $\Phi^+ \subset \Phi$ the set of positive roots included in the set of all roots,
by $s_\alpha$ the simple reflections for $\alpha\in\Delta$,
and by $W$ the Weyl group of $G$.
Recall also the connection index $f$ of $\Delta$,
which equals the determinant of the Cartan matrix of $\Delta$.

For $I\subset \Delta$, we denote by $\Phi_I$, $\Phi^+_I$, $W_I$, and $f_I$ 
the set of roots, positive roots, Weyl group, and the connection index of $I$ respectively.

\subsection{Flag manifolds and Schubert varieties}
\label{sec:Schubert}
The flag manifold $\Fl$ is a projective algebraic manifold 
with a transitive action of $G$ given by left multiplication.
It has a stratification into finitely many $B$-orbits 
(resp. $B^-$-orbits) called the {\em Schubert cells}
$X_w^\circ:= BwB/B$ 
(resp. $X^{w,\circ}:= B^- wB/B $), 
i.e.,
\begin{equation}
\label{schubStrat}
\Fl = \bigsqcup_{w \in W} X_w^\circ =  \bigsqcup_{w \in W} X^{w,\circ} \/. 
\end{equation} 
The closures $X_w:=\overline{X_w^\circ}$ and $X^w:=\overline{X^{w,\circ}}$ are called {\em Schubert varieties}.
The \emph{Bruhat order} is a partial order on $W$ characterized by inclusions of Schubert varieties,
i.e., $X_v \subset X_w$ if and only if $v \le w$,
and $X^w\subset X^v$ if and only if $v\leq w$.
Following \cref{lemma:generate}, the fundamental classes
$\set{[X_v]_T}{v\leq w}$ (resp. $\set{[X^v]_T}{w\leq v}$) form a basis of $H_*^T(X_w)$ (resp. $H_*^T(X^w)$).

The cohomology classes $\sigma_v^T\in H_T^*(X)$ Poincar\'e dual to the $[X^v]_T$,
i.e. characterized by the equation $\sigma_v^T\cap[\Fl]_T=[X^v]_T$,
are called {\em Schubert classes}.
Following \cref{lemma:generate},
the Schubert classes  $\set{\sigma_v^T}{v\in W}$ form a basis of $H_T^*(G/B)$ as a module over $H_T^*(pt)$.

\subsection{Line bundles on the flag manifold}
\label{sec:lineBundles}
Recall that since $G$ is simply connected,
the character group $\X T$ of $T$ equals the weight lattice of $\Delta$.
For $\lambda\in\X T$,
let $\IC_\lambda$ be the one-dimensional $B$-representation on which $T$ acts via the character $\lambda$,
and the unipotent radical of $B$ acts trivially.
We will denote by $\mathcal L_\lambda$ the $T$-equivariant line bundle
\begin{align*}
\mathcal L_\lambda:=G\times^B\IC_{-\lambda}\to G/B,&&(g,v)\mapsto gB,
\end{align*}
with $T$-action given by $t\cdot (g,v)=(tg,v)$.

\subsection{Stability of Dynkin diagrams}
\label{sec:stab}

Let $(\,|\,)$ be a positive-definite $W$-invariant bilinear form on $\g$. 
For $\alpha\in\Delta$, the corresponding coroot is given by $\alpha^\vee=\frac{2(\alpha|\,\_)}{(\alpha|\alpha)}$.
Recall that the Cartan matrix is a square matrix with rows indexed by simple roots,
 with the $\alpha\beta^{th}$ entry given by $a_{\alpha\beta}=\left\langle\beta^\vee,\alpha\right\rangle=\frac{2(\alpha|\beta)}{(\alpha|\alpha)}$.

For $I\subset\Delta$, the Cartan matrix of $I$ is the submatrix of $\Delta$ spanned by  the rows and columns indexed by the roots in $I$.
Let $\Phi^\vee$ denote the set of coroots, and let $\Phi^\vee_I$ denote the subset of coroots corresponding on $I$.
The pairing $\langle\ ,\,\rangle$ on $\Phi_I^\vee\times\Phi_I$ is the restriction of the pairing $\Phi^\vee\times\Phi$
to $\Phi_I\subset\Phi$ and $\Phi^\vee_I\subset\Phi^\vee$.
We describe this by saying that the roots and coroots are stable for the inclusion of Dynkin diagrams.

Consider the elements $\varpi^I_\alpha\in\bigoplus\limits_{\alpha\in I}\IQ\alpha$ 
and $\varpi^{I\vee}_\alpha\in \bigoplus\limits_{\alpha\in I}\IQ\alpha^\vee$
given by the equations
\begin{align*}
&&\langle\varpi^{I\vee}_\alpha,\beta\rangle=\langle\beta^\vee,\varpi^I_\alpha\rangle =\delta_{\alpha\beta}&&\forall\beta\in I.
\end{align*}
Then $\varpi_\alpha:=\varpi_\alpha^\Delta$ is the fundamental weight dual to $\alpha^\vee$,
and $\varpi_\alpha^\vee:=\varpi_\alpha^{\Delta\vee}$ is the fundamental coweight dual to the root $\alpha$.
In general,
\begin{align*}
&&\varpi_\alpha^I\neq\varpi_\alpha&&
\text{and}&& 
\varpi^{I\vee}_\alpha\neq\varpi^\vee_\alpha.&&
\end{align*}
We express this fact by saying that the fundamental weights and coweights are \emph{not stable} for the inclusion of Dynkin diagrams.

\subsection{The height function}
\label{htFn}
Let
\begin{align*}
\rho_I=\frac 12\sum_{\alpha\in\Phi_I^+}\alpha=\sum_{\alpha\in I}\varpi_\alpha^I, && \rho_I^\vee=\frac 12\sum_{\alpha\in\Phi_I^+}\alpha^\vee=\sum_{\alpha\in I}\varpi_\alpha^{I\vee}.
\end{align*}
We set $\rho=\rho_\Delta$, and $\rho^\vee=\rho^\vee_\Delta$.
%
Following \cite[Ch~6, Prop 29]{bourbaki:Lie46},
we have $\langle \rho^\vee,\alpha\rangle=1$ for $\alpha\in\Delta$.
For $\lambda=\sum_{\alpha\in\Delta}a_\alpha\alpha$, 
we define the height of $\lambda$ to be
\begin{align*}
ht(\lambda)=\sum a_\alpha=\langle\rho^\vee,\lambda\rangle.
\end{align*}

Let $h=2\rho^\vee$, and let $\mathfrak s\subset\g$ be the Lie subalgebra spanned by $h$.
Observe that $h$ is in the coroot lattice,
and hence there exists a one-dimensional sub-torus $S\subset T$ with $Lie(S)=\mathfrak s$.
For any $\alpha,\beta\in\Delta$, we have $\left\langle h,\alpha\right\rangle=\left\langle h,\beta\right\rangle$,
and hence $\alpha|\mathfrak s=\beta|\mathfrak s$.
Let $t=\alpha|\mathfrak s$ for some $\alpha\in\Delta$.
The restriction map $\car^*\to\mathfrak s^*$ (dual to the inclusion $\mathfrak s\hookrightarrow\car$)
satisfies $\alpha\mapsto t$ for all $\alpha\in\Delta$,
and hence is given by
$\lambda\mapsto ht(\lambda)t=\left\langle\rho^\vee,\lambda\right\rangle t$.

\section{Hessenberg varieties and Poincar\'e duals}
\label{sec:Hessenberg}
In this section we define Hessenberg varieties
and compute the equivariant Poincar\'e duals (in $G/B$) of Hessenberg varieties corresponding to regular semisimple and regular nilpotent elements.
We also define the Peterson variety \Pet,
and recall from \cite{gms:peterson} some results on the equivariant (co)homology of \Pet.

\subsection{Hessenberg Varieties}
Let $\g:= Lie(G)$, $\bo=Lie(B)$, and $\car:=Lie(T)$.
A subspace $H\subset\g$ is called a \emph{Hessenberg space} if it is $B$-stable and if $\bo\subset H$.
Let
\begin{equation*}
H_0=\bo\oplus\bigoplus\limits_{\alpha\in\Delta}\g_{-\alpha}.
\end{equation*}
We say that a Hessenberg space $H$ is \emph{indecomposable} if $H_0\subset H$.
Recall that the vector bundle $G\times^B\g\to G/B$ is trivialized by the map
\begin{align*}
\mu_\g:G\times^B\g\to\g,&&(g,x)\mapsto Ad(g)x,
\end{align*}
i.e., we have an isomorphism $G\times^B\g\to G/B\times\g$ given by
$(g,x)\mapsto (gB,Ad(g)x)$.
Let $H$ be a Hessenberg space, 
and let $\mu_H$ denote the restriction of $\mu_{\mathfrak g}$ to the sub-bundle $G\times^BH\subset G\times^B\g$.
\begin{center}
\begin{tikzcd}
G\times^BH\arrow[dr,"\mu_H"]\arrow[rr,hook]&& G/B\times\g\arrow[ld,swap,"\mu_\g"]\\
&\g&
\end{tikzcd}
\end{center}
For $x\in\g$, the fibre $\mu_H^{-1}(x)$ (viewed as a subscheme of $G/B$) is called the Hessenberg scheme $\Hess(x,H)$.
If $H$ is indecomposable, $\Hess(x,H)$ is reduced and irreducible for all $x$,
see \cite[Thm 1.2]{abe.fujita.zeng}.
In this case, we call $\Hess(x,H)$ a \emph{Hessenberg variety}.

\emph{
For the rest of this article,
we assume without mention that the Hessenberg space $H$ is indecomposable.
}

For each positive  simple root $\alpha\in \Delta$,
choose a root vector $e_\alpha \in \mathfrak{g}_\alpha$.
Set 
\begin{align}\label{eq:hande}
&e= \sum_{\alpha \in\Delta} e_\alpha,& h=2\rho^\vee.&
\end{align}
The element $e$ is a regular nilpotent element in \bo,
see e.g. \cite{mcgovern.collingwood:nilpotent.orbits,kostant:principal}.
Recall from \cref{htFn} the sub-torus $S\subset T$ lifting the element $h\in\car$.
Following \cite[Ch~6, Prop 29]{bourbaki:Lie46}, we have $\langle\rho^\vee,\alpha\rangle=1$,
and hence
\begin{align*}
[h,e]=[h,\sum_{\alpha\in\Delta} e_\alpha]=\sum_{\alpha\in\Delta} \langle h,\alpha\rangle e_\alpha=2e.
\end{align*}
We see that the vector space $\IC e$ is $h$-stable, and hence also $S$-stable.
Consequently, the Hessenberg variety $\Hess(e,H)$ is $S$-stable.
Since $h\in\car$, the adjoint action of $T$ on $\mathfrak s=\IC h$ is trivial.
In particular, $\mathfrak s$ is $T$-stable, and hence so is $\Hess(h,H)$.
The following result was proved in type A by Anderson and Tymozcko in \cite{MR2609167}.

\begin{thm}
\label{HessClass}
For any indecomposable Hessenberg space $H$, we have
\begin{align*}
[\Hess(h,H)]_T&=\prod\left(c_1^T(\mathcal L_\alpha)\right)\cap[\Fl]_T,\\
[\Hess(e,H)]_S&=\prod\left(c_1^S(\mathcal L_\alpha)-t\right)\cap[\Fl]_S,
\end{align*}
where the product is over the set $\set{\alpha\in\Phi^+}{\g_{-\alpha}\not\subset H}$.
\end{thm}
\begin{proof}
Consider the vector bundle
$
\mathcal V=G\times^B(\g/H)\to G/B,
$
which admits a filtration with quotient bundles
$
\set{\mathcal L_\alpha}{\alpha\in\Phi^+,\g_{-\alpha}\not\subset H}.
$
For $x$ a regular element of \g,
let $s_x:G/B\to\mathcal V$ be the section of $\mathcal V$ given by $s_x(gB)=(g, Ad(g^{-1})x)$.
Following \cite[Prop 3.6]{abe.fujita.zeng},
we have $\Hess(x,H)=Z(s_x)$, the zero scheme of $s_x$.

Observe that $\mathcal V$ is a $T$-equivariant vector bundle,
and $s_h$ is a $T$-invariant section.
Therefore by \cref{lem:eqSec},
the fundamental class of $\Hess(h,H)=Z(s_h)$ is given by the first equality.
On the other hand, the section $s_e$ lies in the $t$-eigenspace of the $S$-action on $H^0(G/B,\mathcal V)$,
hence the second equality holds for the fundamental class of $\Hess(e,H)=Z(s_e)$ by \cref{cor:twistedSection}.
\end{proof}

\subsection{The Peterson variety}
\label{subsec:petersonVariety}
The Peterson variety is defined by
\begin{align}
\label{defn:Pet}
\Pet:= \Hess(e,H_0)\subset\Fl.
\end{align}
It is a subvariety of \Fl\ of dimension $rk(G)=|\Delta|$, singular in general.
Following \cref{HessClass}, the $S$-equivariant fundamental class of the Peterson variety in $H_*^S(G/B)$ is given by
\begin{align}
\label{PetClass}
[\Pet]_S=\prod\limits_{\alpha\in\Phi^+\backslash\Delta}\left(c_1^S\left(\mathcal L_\alpha\right)-t\right)\cap[\Fl]_S.
\end{align}

Let $\Pet_I$ denote the Peterson variety corresponding to the Dynkin diagram $I\subset\Delta$.
The following was proved in classical types by Tymoczko \cite[Thm~4.3]{tymoczko:paving}
and generalized to all Lie types by Precup \cite{precup},
see also \cite[Appendix A]{gms:peterson}.

\begin{prop}\label{prop:vectorSpace}
There exists a natural embedding $\Pet_I\subset\Pet$,
and a corresponding $S$-stable affine paving $\Pet=\bigsqcup\limits_{I\subset\Delta}\Pet_I^\circ$,
where $\Pet_I^\circ=\Pet_I\backslash\bigcup\limits_{J\subsetneq I}\Pet_J$.
\end{prop}
The subvarieties $\Pet^\circ_I$ (called \emph{Peterson cells}) are $S$-stable affine spaces.
The Peterson cell $\Pet^\circ_I$ has a unique $S$-stable point $w_I$,
the longest element in the Weyl subgroup $W_I$.
Following \cref{lemma:generate},
the fundamental classes $\set{[\Pet_I]_S}{I\subset\Delta}$ form a basis of $H_*^S(\Pet)$ over $H^*_S(pt)$.

\begin{prop}
{\rm(\cite[Thm 4.3]{gms:peterson})}
\label{prop:duality}
Consider the inclusion $i:\Pet\hookrightarrow\Fl$.
For each $I\subset\Delta$, fix a Coxeter element $v_I$ for $I$.
There exist positive integers $m(v_I)$ such that
\begin{align*}
\left\langle i^*\sigma_{v_I}^S,[\Pet_J]_S\right\rangle=m(v_I)\delta_{IJ}.
\end{align*}
In particular, 
$\set{i^*\sigma_{v_I}^S}{I\subset \Delta}$ is a basis for $H^*_S(\Pet)$.
Furthermore, the numbers $m(v_I)$ do not depend on the superset $\Delta$ containing $I$.
\end{prop}

\Cref{stab1} deals with the stability of Schubert classes and their pullbacks to the Peterson variety.
For $I\subset\Delta$, let $G_I\subset G$ be the standard Levi subgroup,
and $B_I:=B\cap G_I$ the corresponding Borel subgroup of $G_I$.

\begin{prop}
{\rm(\cite[Thm 6.6]{gms:peterson})}
\label{stab1}
Consider the inclusions $\iota_I:\Pet_I\hookrightarrow\Pet$
and $i_I:\Pet_I\hookrightarrow G_I/B_I$.
For $w\in W$, let $p_w=i^*\sigma_w^S$.
For $w\in W_I$, let $p^I_w=i_I^*\sigma_w^S$.
Then $\iota_I^*p_w=p_w^I$.
\end{prop}

For $w\in W_I$, \cref{stab1} allows us to abuse notation and denote by $p_w$ both
the class $i^*\sigma_w^S$ in $H_S^*(\Pet)$ and its pullback $p_w^I$ in $H^*_S(\Pet_I)$.

\section{The Equivariant Giambelli formula and intersection multiplicities}
\label{sec:giambelli}
In this section, 
we first recall the relationship between $H^*(\Fl)$, $H^*(\Pet)$, and $H^*(\Perm)$,
which we then use to compute the multiplicities $m(v_I)$ in \cref{prop:duality},
and to develop an equivariant Giambelli formula for \Pet,
i.e., a formula expressing the pullback of Schubert classes as a polynomial in the divisor classes.


\subsection{Cohomology of regular Hessenberg varieties}
\label{sec:bc}
The goal of Section~\ref{sec:bc} is primarily expository.
Let $H$ be any (indecomposable) Hessenberg space, and let $h$ and $e$ be as in \cref{eq:hande}.
Klyachko \cite{klyachko85,klyachko95} and Tymoczko \cite{tymoczko2007permutation} have constructed an action of $W$ on $H^*(\Hess(h,H))$,
called the Tymoczko dot-action. 
Abe, Horiguchi, Masuda, Murai, Sato \cite{MR4116638} proved a commutative diagram relating the cohomologies of $G/B$ and certain Hessenberg varieties (see \cref{triangle}).
This relationship between the cohomologies of the regular nilpotent and regular semisimple Hessenberg varieties was earlier described in type A by Abe, Harada, Horiguchi, and Masuda in \cite{ahhm:cohomologyring}.
A proof of \cref{triangle} for $s$ regular, semisimple  in a Euclidean neighborhood of $e$ can be found in the works of Brosnan and Chow \cite{brosnan-chow} and
B\u alibanu and Crooks \cite{balibanu2020perverse}. 
We  show that the cohomology of the  regular nilpotent Hessenberg variety is isomorphic to the Weyl-invariant part of the cohomology of the regular semisimple Hessenberg variety, $\Hess(h, H)$.

Recall the map $\mu_H:G\times^BH\to\g$ given by $(g,x)\mapsto Ad(g)x$.
Let $\g^r$ denote the set of regular elements in \g, and let $H^r=H\cap \g^r$.
Following \cite[Sec 4]{balibanu2020perverse},
for any $x\in\g$,
there exists a Euclidean open neighbourhood $D_x$ of $x$,
such that the (non-equivariant) inclusion $\mu_H^{-1}(x)\hookrightarrow\mu_H^{-1}(D_x)$ induces an isomorphism
\begin{equation}
\label{DxIso}
H^*(\mu_H^{-1}(D_x))\overset\sim\longrightarrow H^*(\mu_H^{-1}(x)).
\end{equation}
Let $Z=\mu_H^{-1}(D_x)\cap(G\times^BH^r)$,
let $s$ be a regular semisimple element contained in $D_x$, 
and consider the commutative diagram
\begin{equation}
\label{bigDiag}
\begin{tikzcd}
G\times^BH\arrow[d]\arrow[r,hookleftarrow]&\mu_H^{-1}(D_x)\arrow[r,hookleftarrow]&Z\arrow[r,hookleftarrow]&\Hess(s,H)\\
G/B\arrow[r,hookleftarrow]&\Hess(x,H)\arrow[u,hook].
\end{tikzcd}
\end{equation}
Composing the induced pullback map $H^*(\mu_H^{-1}(D_x))\to H^*(\Hess(s,H))$
with the isomorphism \eqref{DxIso},
we obtain the so-called local invariant cycle map
\begin{align*}
\lambda_x:H^*(\Hess(x,H))\to H^*(\Hess(s,H))^W.
\end{align*}

Following \cite{balibanu2020perverse},
an application of the local invariant cycle theorem of Beilinson, Bernstein, and Deligne \cite{MR751966},
yields the following result. 

\begin{prop}
\label{bbd}
The map $\lambda_x:H^*(\Hess(x,H))\to H^*(\Hess(s,H))^W$ is surjective.
\end{prop}

\begin{cor}
Consider the inclusion $f:\Hess(s,H)\to G/B$.
The image of the pullback $f^*:H^*(\Fl)\to H^*(\Hess(s,H))$ is precisely $H^*(\Hess(s,H))^W$.
\end{cor}
\begin{proof}
We apply \cref{bbd} to $x=0$, in which case $\Hess(x,H)=G/B$,
and
$\lambda_0$ is precisely the pullback for the inclusion $\Hess(s,H)\hookrightarrow G/B$.
\end{proof}

\begin{prop}
\label{triangle}
We have a commutative diagram,
\begin{equation}
\label{commTri}
\begin{tikzcd}
&H^*(\Fl)\arrow[dl,swap,"i^*",twoheadrightarrow]\arrow[d,twoheadrightarrow,dashed]\arrow[dr,"j^*"]&\\
H^*(\Hess(e,H))\arrow[r,"\sim"]&H^*(\Hess(h,H))^W\arrow[hook,r]&H^*(\Hess(h,H)).
\end{tikzcd}
\end{equation}
\end{prop}
\begin{proof}
For $x=e$, there exists a regular semisimple element $s$ in a neighborhood of $e$ such that the map $\lambda_e:H^*(\Hess(e,H))\to H^*(\Hess(s,H))^W$ is an isomorphism,
cf. \cite[Prop. 4.7]{balibanu2020perverse}.
 Observe that $\Hess(e,H)\hookrightarrow Z$ in this case. Thus we have the following commutative diagram
\begin{center}
\begin{tikzcd}
&H^*(Z)\arrow[dl]\arrow[dr]&\\
H^*(\Hess(e,H))\arrow[rr,"\lambda_e"]&&H^*(\Hess(s,H))^W.
\end{tikzcd}
\end{center}

Following Diagram \eqref{bigDiag}, the pullbacks $i^*:H^*(G/B)\to H^*(\Hess(e,H))$ and 
 $j^*:H^*(G/B)\to H^*(\Hess(s,H))$ 
factor through the pullback $H^*(G/B)\to H^*(Z)$,
and hence we have a commutative diagram
\begin{equation}
\label{work2}
\begin{tikzcd}
&H^*(\Fl)\arrow[dl,swap,"i^*",twoheadrightarrow]\arrow[d,twoheadrightarrow,dashed]\arrow[dr,"j^*"]&\\
H^*(\Hess(e,H))\arrow[r,"\sim"]&H^*(\Hess(s,H))^W\arrow[hook,r]&H^*(\Hess(s,H)).
\end{tikzcd}
\end{equation}

We conjugate $s$ to a regular semisimple element $s'\in T$. Note that the corresponding Hessenberg varieties $\Hess(s, H)$ and $\Hess(s',H)$ are isomorphic varieties, but they have conjugate torus actions. Since  $s'$ and $h$ are both regular and semisimple in $T$, we invoke the isomorphism described in \cite{MR4116638} between $H_T^*(\Hess(s',H))$ and $H^*_T(\Hess(h,H))$ that descends to an isomorphism on the ordinary cohomology:
\begin{equation*}
\begin{tikzcd}[sep=small]
&&H^*_T(G/B)\arrow[dl]\arrow[dd]\arrow[dr]&&\\
&H^*_T(\Hess(s',H))\arrow[dd]\arrow[rr,"\hspace{-.3in}\sim"] &&H_T^*(\Hess(h,H))\arrow[dd]&\\
&&H^*(G/B)\arrow[dl]\arrow[dr]&&\\
H^*(\Hess(s, H))\arrow[r, equal]&H^*(\Hess(s',H))\arrow[rr,"\sim"]&&H^*(\Hess(h,H)).&
\end{tikzcd}
\end{equation*}
Here the vertical maps are each the surjective maps naturally induced by forgetting the $T$ action. 
The bottom horizontal isomorphism is invariant with respect to the $W$-action on the cohomology, inducing an isomorphism between $H^*(\Hess(s,H))^W$ and $H^*(\Hess(h,H))^W$. The commuting diagram \eqref{commTri} then follows by the commuting diagram~\eqref{work2} together with these isomorphisms.
\end{proof}

\subsection{Giambelli and Multiplicity Formulae}

 We now specialize to the case  $\Hess(e,H)=\Pet$ and $\Hess(h,H)=\Perm$.
 There are two parts to the proof of the equivariant Giambelli formula, namely showing that there is a single non-zero coefficient, and computing this coefficient.
The calculation of the non-zero coefficient is deduced from the analogous result in \cite{klyachko85} for the ordinary cohomology of the permutahedral variety.
The claim that there a single non-zero coefficient is essentially from \cite{drellich:monk}, where the result is proved only for specific choices of $v_K$, but the proofs follow almost verbatim.

\begin{lemma}[Ordinary Giambelli Formula]
\label{klyachkoFormula}
Let $v_I$ be a Coxeter element for $I\subset\Delta$,
and let $R(v_I)$ be the number of reduced words for $v_I$.
We have
\begin{equation}
i^*\sigma_{v_I}=\frac{R(v_I)}{|I|!}\prod_{\alpha\in I}i^*\sigma_\alpha.
\end{equation}
\end{lemma}
\begin{proof}
Let $\mathcal R(v_I)$ denote the set of reduced words for $v_I$.
Following \cite{klyachko85,klyachko95} (see also \cite[Thm 8.1]{nadeau.tewari}),
we have
\begin{align*}
j^*\sigma_{v_I}=\frac{1}{\ell(v_I)!}\sum_{\underline v\in\mathcal R(v_I)}\prod_{s_\alpha\in\underline v}j^*\sigma_\alpha,
\end{align*}
where $\ell(\_)$ denotes the length function on $W$.
Since $v_I$ is a Coxeter word for $I$,
we have $\ell(v_I)=|I|$.
Further,
every reduced word $\underline v\in\mathcal R(v_I)$ contains each simple reflection $\set{s_\alpha}{\alpha\in I}$ exactly once.
Hence we have,
\begin{align}
\label{work4}
\frac{1}{\ell(v_I)!}\sum_{\underline v\in\mathcal R(v_I)}\prod_{s_\alpha\in\underline v}j^*\sigma_\alpha
=\frac{R(v_I)}{|I|!}\prod_{\alpha\in I}j^*\sigma_\alpha.
\end{align}
Finally, it follows from \cref{commTri} 
that the pullbacks $i^*:H^*(G/B)\to H^*(\Pet)$ and $j^*:H^*(G/B)\to H^*(\Perm)$ have the same kernel,
and hence \emph{any relation amongst the classes $j^*\sigma_w$ also holds amongst the classes $i^*\sigma_w$}.
The claim now follows from \cref{work4}.
\end{proof}

Recall that we denote by $p_w$ the pullback class $i^*\sigma_w^S$,
and by $p_\alpha$ the pullback class $i^*\sigma_{\alpha}^S$.
For convenience, we write the equivariant class
\begin{align*}
\Omega_I:=\prod_{\alpha\in I}p_\alpha = \prod_{\alpha\in I}i^*(\sigma^S_{\alpha})
\end{align*}
for any $I\subset\Delta$, where $i^*: H_S^*(G/B)\rightarrow H^*_S(\Pet)$ is the pullback map in equivariant cohomology induced by the inclusion $i:\Pet\to G/B$.

\begin{thm}[Equivariant Giambelli formula]
\label{Giambelli}
Let $v_I$ be a Coxeter element for $I\subset\Delta$,
and let $R(v_I)$ be the number of reduced words for $v_I$.
We have
\begin{align*}
p_{v_I}=\frac{R(v_I)}{|I|!} \Omega_I.
\end{align*}
\end{thm}
\begin{proof}
Observe that the restriction to ordinary cohomology $H^*_S(\Pet)\to H^*(\Pet)$
is given by 
\begin{align*}
\Omega_J\mapsto \prod_{\alpha\in J}i^*\sigma_\alpha.
\end{align*}
Consider $J\subseteq \Delta$.
Following \cref{klyachkoFormula,prop:duality} 
$$\set{\prod_{\alpha\in K}i^*\sigma_\alpha}{K\subseteq J}$$ 
is a basis of $H^*(\Pet_J)$.
Since $H^*_S(\Pet_J)$ (\cref{lemma:generate}) is a free module over $\IQ[t]$, by equivariant formality, $\set{\Omega_K}{K\subseteq J}$ is a basis of $H^*_S(\Pet_J)$ over $\IQ[t]$.

Consider now the basis expansion 
\begin{align}
\label{pvExpansion}
p_{v_I}=i^*\sigma_{v_I }^{S}=\sum_{J\subseteq\Delta} c_{v_I }^J\Omega_J
\end{align}
in $H^*_S(\Pet)$,
where the $c_{v_I}^J\in\IQ[t]$ are $t$-monomials of degree $|I|-|J|$.

First consider $J$ for which $I\subsetneq J$.
Then $|I|-|J|<0$, and hence $c_{v_I}^J= 0$.

Next, consider the case $J=I$.
The coefficient $c_{v_I}^I$ has degree $0$, hence can be obtained from the specialization $H^*_S(\Pet)\to H^*(\Pet)$.
Applying \cref{klyachkoFormula}, 
\begin{align*}
c_{v_I }^I=\frac{R(v_I)}{|I|!}.
\end{align*}

Finally, suppose $J$ satisfies $I\not\subseteq J$.
Consider the localization of $\sigma_{v_I}$ to the fixed point $w$.
Then $\sigma_{v_I}|_w\neq0$ implies $v_I\leq w$ in the Bruhat order, see \cite{billey:kostant}. 
In particular, for all fixed points $w\in \Pet_J$, we have $\sigma_{v_I}|_w=0$. 
By injectivity of localization (see \cite[Thm. 3.1]{drellich:monk}), the pull-back of $\sigma_{v_I}$ to $\Pet_J$ is zero.
Recall that the pull-back $H^*_S(\Pet)\to H^*_S(\Pet_J)$ sends $\Omega_K$ to $\Omega_K$ if $K\subset J$, and to $0$ otherwise.
Applying the pull-back, \Cref{pvExpansion} yields
$$
0 = \sum_{K\subseteq J} c_{v_I }^K \Omega_K
$$
in $H^*_S(\Pet_J)$.
Since the set $\set{\Omega_K}{K\subset J}$ forms a basis of $H^*_S(\Pet_J)$, we deduce that $c_{v_I}^K=0$ for all $K\subset J$;
in particular, $c_{v_I}^J=0$.
\end{proof}

\begin{lemma}
\label{lem:mult}
We have 
$
\left\langle\Omega_\Delta,[\Pet]_S\right\rangle=\frac{|W|}{f_\Delta},
$
where $f_\Delta$ is the connection index of $\Delta$.
\end{lemma}
\begin{proof}
Observe that since $\deg\left([\Pet]_S\right)=\deg\left(\Omega_\Delta\right)=|\Delta|$,
we have 
$
\left\langle\Omega_\Delta,[\Pet]_S\right\rangle=
\left\langle\prod i^*\sigma_{\alpha},[\Pet]\right\rangle,
$
where the latter expression is the pairing in ordinary (co)homology.  Observe that  $[\Pet]=[\Perm]$ in $H_*(\Fl)$ by \cite{abe.fujita.zeng}. (Alternatively, apply
 the forgetful maps $H^*_T(G/B)\to H^*(G/B)$ and $H^*_S(G/B)\to H^*(G/B)$ to the equations in \cref{HessClass} to draw the same conclusion.)
It follows that
\begin{equation*}
\left\langle\Omega_\Delta,[\Pet]_S\right\rangle=
\left\langle\prod_{\alpha\in\Delta} i^*\sigma_\alpha,[\Pet]\right\rangle=
\left\langle\prod_{\alpha\in\Delta} i^*\sigma_\alpha,[\Perm]\right\rangle=\frac{|W|}{f_\Delta},
\end{equation*}
where the latter equality is from \cite[Thm 3]{klyachko85}.
\end{proof}


\begin{thm}[Multiplicity Formula]
\label{thm:mult}
For $v_I$ a Coxeter element of $I$,
we have
\begin{equation*}
\left\langle p_{v_I},[\Pet_I]_S\right\rangle=\frac{R(v_I)|W_I|}{|I|!\,f_I}.
\end{equation*}
\end{thm}
\begin{proof}
Since $m(v_I)$ does not depend on the diagram $\Delta$ containing $I$,
we may assume $I=\Delta$.
%
Using \cref{Giambelli,lem:mult}, we obtain
\begin{equation*}
\left\langle p_{v_I},[\Pet_I]_S\right\rangle
=\frac{R(v_I)}{|I |!}\left\langle\Omega_I ,[\Pet_I]\right\rangle=\frac{R(v_I)|W_I|}{|I |!f_I}.
\end{equation*}
\end{proof}

\section{Dual Peterson Classes and Chevalley and Monk formulae}
\label{sec:chevalley}
In this section,
we present a Chevalley formula for the cap product of a divisor class with a fundamental class $[\Pet_I]_S$,
and a Monk rule with respect to the basis $\set{\Omega_I}{I\subset\Delta}$.
The Monk rule is dual to the Chevalley formula;
the precise relationship between the two follows from \cref{prop:duality,thm:mult}.
We also recover the presentation, obtained by Harada, Horiguchi, and Masuda in \cite{harada.horiguchi.masuda:Peterson},
of $H^*_S(\Pet)$ as a quotient of $H^*_S(\Fl)$,

Recall the Schubert classes $\sigma^S_\alpha$ from \cref{sec:Schubert},
and the line bundles $\mathcal L_\lambda\to\Fl$ from \cref{sec:lineBundles}.
For $\alpha,\beta\in\Delta$, let $a_{\alpha\beta}=\left\langle\beta^\vee,\alpha\right\rangle$ be the $\alpha\beta^{th}$ entry of the Cartan matrix of $\Delta$.
For $\alpha\in\Delta$,
we set $p_\alpha=i^*\sigma_\alpha^S$ and $q_\alpha=\sum_{\beta\in\Delta}a_{\alpha\beta} p_\beta$,
where $i$ denotes the embedding $i:\Pet\hookrightarrow\Fl$.

\subsection{Dual Peterson Classes}
In \cref{qChevalley}, we compute the equivariant cohomology class in $H^*_S(\Pet)$ which is dual to the Peterson subvariety $\Pet_I$.
This is a key step in our proof of the Chevalley formula.

\begin{lemma}
\label{qalpha}
For $\alpha\in\Delta$, we have $c_1^S(i^*\mathcal L_\alpha)=q_\alpha-t$.
\end{lemma}
\begin{proof}
Let $\varpi_\alpha$ be the fundamental weight dual to the coroot $\alpha^\vee$,
let $V_{\varpi_\alpha}$ be the corresponding irreducible $G$-representation,
and let $\pr:V_{\varpi_\alpha}\to\IC_{-\varpi_\alpha}$
the $B$-equivariant projection onto the lowest weight space in $V_{\varpi_\alpha}$. 

Let $\mathbf 1\in\IC_{-\varpi_\alpha}$ be a lowest weight vector in $V_{\varpi_\alpha}$,
and consider the section $s$
of the line bundle $\mathcal L_{\varpi_\alpha}\to G/B$
given by  $s(gB)= (g,\pr(g^{-1}\mathbf 1))$.
Observe that the torus $T$ acts on $s$ via the character $-\varpi_\alpha$,
\begin{align*}
(z\cdot s)(gB) & =zs(z^{-1}gB)  = z(z^{-1}g,\pr(g^{-1}z\mathbf 1))\\
               & =z(z^{-1}g,\varpi_\alpha(z^{-1})\pr(g^{-1}\mathbf 1))\\
               & =\varpi_\alpha(z^{-1}) (g,\pr(g^{-1}\mathbf 1)) =\varpi_\alpha(z^{-1})s(gB).
\end{align*}
Further, the zero scheme $Z(s)$ of $s$ is supported precisely on the Schubert divisor $X^{s_\alpha}$,
thus $[Z(s)]_T=m[X^{s_\alpha}]_T$ for some positive integer $m$.
It follows from \cref{cor:twistedSection} that
\begin{align}
\label{work1}
m\sigma^T_\alpha=c_1^T(\mathcal L_{\varpi_\alpha})+\varpi_\alpha.
\end{align}
We evaluate \cref{work1} under localization
$\ell_{s_\alpha}^*:H^*_T(G/B)\to H^*(\{s_\alpha\})$ at the $T$-fixed point $s_\alpha$.
Recall that $\mathcal L_{\varpi_\alpha}=G\times^B\IC_{-\varpi_\alpha}$,
and hence $\ell_{s_\alpha}^*(c_1(\mathcal L_{\varpi_\alpha}))=s_\alpha(-\varpi_\alpha)=-\varpi_\alpha+\alpha$. 
Using the localization formula of Andersen, Jantzen, and Soergel \cite{AJS}, and Billey \cite{billey:kostant},
we have $\ell_{s_\alpha}^*(\sigma_\alpha)=\alpha$.
It follows that $m=1$, i.e.,
\begin{align*}
c_1^T(\mathcal L_{\varpi_\alpha})=\sigma^T_\alpha-\varpi_\alpha.
\end{align*}
Recall that $\alpha=\sum a_{\alpha\beta}\varpi_\beta$,
where $a_{\alpha\beta}=\left\langle\beta^\vee,\alpha\right\rangle$
is the $\alpha\beta^{th}$ entry of the Cartan matrix.
We have
\begin{align*}
c_1^T(\mathcal L_\alpha)
=\sum a_{\alpha\beta} c_1^T(\mathcal L_{\varpi_\beta})
=\sum a_{\alpha\beta}\sigma^T_\beta-\sum a_{\alpha\beta}\varpi_\beta
=\sum a_{\alpha\beta}\sigma^T_\beta-\alpha.
\end{align*}
Consequently, we have
$
c_1^S(\mathcal L_\alpha)
=\sum a_{\alpha\beta}\sigma^S_\beta-t. 
$
Applying the $S$-equivariant pullback $i^*$, we obtain the claimed equality,
$c_1^S(i^*\mathcal L_\alpha)=q_\alpha-t$.
\end{proof}

\begin{prop}
\label{eulerClassGP}
For $I\subset\Delta$,
we have the following equality in $H_*^S(G/B)$:
\begin{align*}
\prod\limits_{\alpha\in\Phi^+\backslash\Phi_I^+}\left(c_1^S(\mathcal L_\alpha)-t\right)\cap[G/B]_S= \frac{|W|}{|W_I|}[X_{w_I}]_S.
\end{align*}
\end{prop}
\begin{proof}
Let $P\subset G$ be the parabolic subgroup corresponding to $I$,
and let $\mathfrak p=Lie(P)$.
Recall that the tangent bundle $T(G/P)$ has the following description:
$T(G/P)=G\times^P(\g/\mathfrak p)$.
Let $\pr:\g\to\g/\mathfrak p$ denote the projection,
and consider the section $s:G/P\to T(G/P)$ given by
\begin{align*}
s(gP)=(g,\pr(Ad(g^{-1})e)).
\end{align*}
The zero scheme of $s$ is supported at the single point $1.P$,
and $S$ acts on $s$ via the character $t$.
It follows from \cref{cor:twistedSection} that there exists an integer $N$ such that
\begin{align}
\label{eq:EulerChar}
N[1.P]_S = e^S(T(G/P)\otimes\underline\IC_{-t})\cap [G/P]_S.
\end{align}
Let $\chi(\_)$ denote the Euler characteristic.
Mapping \cref{eq:EulerChar} to ordinary cohomology,
we obtain 
$$
N = e(T(G/P))\cap [G/P]  = \chi(G/P),
$$ 
where the second equality follows from the Poincar\'e-Hopf index theorem, see for example \cite[Ch. 3]{MR0348781}. 
Moreover, since the Euler characteristic of $G/P$ equals the number of Schubert cells in $G/P$,
we have $N = \frac{|W|}{|W_I|}$.
Pulling back \cref{eq:EulerChar} along the ($S$-equivariant) flat map $\pi:G/B\to G/P$,
see \cite{edidin.graham},
we have
\begin{align*}
\frac{|W|}{|W_I|}[X_{w_I}]_S=e^S(\pi^*T(G/P)\otimes\underline\IC_{-t})\cap [G/B]
\end{align*}

Finally, we observe that $\pi^*T(G/P)=G\times^B\mathfrak g/\mathfrak p$
has a filtration with quotients 
$\set{\mathcal L_\alpha}{\alpha\in\Phi^+\setminus \Phi_I^+}$, so that
$$
e^S(\pi^*T(G/P)\otimes \underline{\mathbb C}_{-t}) = \prod_{\alpha\in\Phi^+\backslash\Phi_I^+ } (c^S_1(\mathcal L_\alpha)-t).
$$ The result of the proposition now follows after capping with $[G/B]$.
\end{proof}

Using \cref{eulerClassGP}, we can compute cohomology classes in $H^*_S(\Pet)$ that are dual to the Peterson subvarieties.
\begin{thm}
\label{qChevalley}
For any $I\subset\Delta$, we have 
\begin{align*}
\prod\limits_{\alpha\in\Delta\backslash I}(q_\alpha-2t)\cap[\Pet]_S=\frac{|W|}{|W_I|}[\Pet_I]_S.
\end{align*}
\end{thm}
\begin{proof}

It is sufficient to prove the result in the case where $I=\Delta\backslash\{\alpha\}$ for some $\alpha\in\Delta$.
Let $G_I$ be the standard Levi subgroup corresponding to $I\subset\Delta$, and let $B_I=G_I\cap B$.
We identify $G_I/B_I$ with the Schubert variety $X_{w_I}$.
For convenience, we write $r_\alpha=c_1^S(\mathcal L_\alpha)-t$.
We have (in $H^S_*(G/B)$)
\begin{align*}
r_\alpha\cap[\Pet]_S
&=r_\alpha\prod\limits_{\beta\in\Phi^+\backslash\Delta} r_\beta\cap [G/B]_S
&&\text{by \cref{PetClass}}\\
&= \prod\limits_{\beta\in\Phi_I^+\backslash I}r_\beta\prod\limits_{\beta\in\Phi^+\backslash\Phi_I^+} r_\beta\cap [G/B]_S\\
&=\frac{|W|}{|W_I|} \prod\limits_{\beta\in\Phi_I^+\backslash I}r_\beta\cap[G_I/B_I]_S
&&\text{by \cref{eulerClassGP}}\\
&=\frac{|W|}{|W_I|}[\Pet_I]_S
&&\text{by \cref{PetClass}}.
\end{align*}
Following \cref{qalpha}, we have $i^*r_\alpha=q_\alpha-2t$.
The result follows from the projection formula (\cref{projectionFormula}) applied to the inclusion $i:\Pet\to G/B$.
\end{proof}

As an application of \cref{qChevalley},
we recover the presentation of $H^*_S(\Pet)$ as a quotient of $H^*_S(\Fl)$,
first obtained by Harada, Horiguchi, and Masuda \cite{harada.horiguchi.masuda:Peterson}.
\begin{cor}
\label{hhm}
Recall that $q_\alpha=\sum_{\beta\in\Delta}\left\langle\beta^\vee,\alpha\right\rangle p_\beta$.
The equivariant cohomology ring of the Peterson variety admits the presentation 
\begin{align*}
H^*_S(\Pet)=\frac{\mathbb Q[p_\alpha]_{\alpha\in\Delta}}{\left\langle p_\alpha(q_\alpha-2t)\right\rangle}_{\alpha\in\Delta}.
\end{align*}
\end{cor}
\begin{proof}
 Following \cite{drellich:monk}, the map $H^*_S(G/B)\to H^*_S(\Pet)$ is surjective.
Recall that $H^*_T(G/B)$ is generated (as a ring) by the divisor classes.
Consequently, the $\set{p_\alpha}{\alpha\in\Delta}$ are ring generators for $H^*_S(\Pet)$. 

Let $I=\Delta\backslash\{\alpha\}$. 
Following \cref{qChevalley} and \cite[Thm 6.5, 6.6]{gms:peterson}, we have
\begin{align*}
(q_\alpha-2t)\cap[\Pet]_S & =\frac{|W|}{|W_I|} [\Pet_I]_S,\\
p_\alpha\cap [\Pet_I]_S   & =0,
\end{align*}
and hence
$
p_\alpha(q_\alpha-2t)\cap[\Pet]_S=0.
$
By the universal coefficients theorem \cite[Ch 17, \S 3]{MR1702278}, the map 
$\omega\mapsto\omega\cap[\Pet]_S$ is an isomorphism $H^*_S(\Pet)\xrightarrow\sim H_*^S(\Pet)$.
It follows that $p_\alpha(q_\alpha-2t)=0$.
Consequently, we obtain a surjective map
\begin{align}
\label{surjRelations}
\frac{\mathbb Q[p_\alpha]_{\alpha\in\Delta}}{\left\langle p_\alpha(q_\alpha-2t)\right\rangle}_{\alpha\in\Delta}\longrightarrow H^*_S(\Pet).
\end{align}

It remains to show that this map is an isomorphism, i.e., there are no other relations.   \cref{triangle} specialized to the case of the Peterson and permutohedral varieties, states that $H^*(\Pet)\cong H^*(\Perm)^W$. Using Klyachko's presentation \cite[Thm 3]{klyachko85} of the latter's cohomology, 
we have
\begin{align*}
H^*(\Pet)=\dfrac{\mathbb Q[p_\alpha]_{\alpha\in\Delta}}{\left\langle p_\alpha q_\alpha\right\rangle_{\alpha\in\Delta}}.
\end{align*}
We deduce that the kernel of the map in \eqref{surjRelations} is $t$-divisible.
Since $H^*_S(\Pet)$ is torsion free over $H^*_S(pt)$,
we conclude that \eqref{surjRelations} is an isomorphism.
\end{proof}

\subsection{The Chevalley and Monk formulae}
\begin{thm}[Equivariant Chevalley formula]
\label{thm:chevalley}
For $\alpha\in\Delta$, $J\subset\Delta$, we have
\begin{align*}
p_\alpha\cap[\Pet_J]_S=
\begin{cases}
0 &\text{if }\alpha\not\in J,\\
\left\langle2\rho_J^\vee,\varpi^J_\alpha\right\rangle t\,[\Pet_J]_S+
 \sum\limits_{\beta\in J}
\left\langle\varpi^{J\vee}_\beta,\varpi^J_\alpha\right\rangle
\frac{|W_J|}{|W_{J\backslash \{\beta\}}|}[\Pet_{J\backslash \{\beta\}}]_S
&\text{if }\alpha\in J.
\end{cases}
\end{align*}
\end{thm}
\begin{proof}
Recall the inclusion map $\iota_J:\Pet_J\hookrightarrow\Pet$.
Following \cite[Thm. 6.6(b)]{gms:peterson}, we have 
\begin{equation*}\iota_J^*p_\alpha=
\begin{cases}
p_\alpha & \text{if }\alpha\in J,\\
0        & \text{otherwise}.
\end{cases}
\end{equation*}
In particular, we have $p_\alpha\cap[\Pet_J]_S=0$ for $\alpha\not\in J$.

Next, for $\alpha\in J$, we have 
$
\varpi^J_\alpha = \sum_{\beta\in J} \left\langle \varpi^{J\vee}_\beta,\varpi^J_\alpha\right\rangle \beta,
$
and hence
\begin{align*}
&&\iota_J^*p_\alpha =\sum_{\beta\in J} \left\langle \varpi^{J\vee}_\beta,\varpi^J_\alpha\right\rangle \iota_J^* q_\beta 
&&\text{in }H^*_S(\Pet_J).
\end{align*}
Following \cref{qChevalley}, we have
\begin{align*}
&&q_\beta\cap[\Pet_J]_S
=\frac{|W_J|}{|W_{J\backslash\{\beta\}}|}[\Pet_{J\backslash\{\beta\}}]_S+2t[\Pet_J]_S.
&&
\end{align*}
Further, for $\beta\in J$, we have $\iota_J^*q_\beta=\sum_{\alpha\in J}a_{\beta\alpha}\iota_J^*p_\alpha=\sum_{\alpha\in J}a_{\beta\alpha}p_\alpha=q_\beta$.
Hence, by the projection formula (\cref{projectionFormula}),
\begin{align*}
p_\alpha\cap[\Pet_J]_S&=\sum_{\beta\in J}\langle \varpi^{J\vee}_\beta,\varpi^J_\alpha\rangle\, q_\beta \cap [\Pet_J]_S \\
&=
\sum_{\beta\in J}
\left\langle\varpi^{J\vee}_\beta,\varpi^J_\alpha\right\rangle
\left(
2t [\Pet_J]_S+
\frac{|W_J|}{|W_{J\backslash\{\beta\}}|}[\Pet_{J\backslash\{\beta\}}]_S\right)
\\
&=
\left\langle2\rho_J^\vee,\varpi^J_\alpha\right\rangle t
[\Pet_J]_S+
\sum_{\beta\in J}
\left\langle\varpi^{J\vee}_\beta,\varpi^J_\alpha\right\rangle
\frac{|W_J|}{|W_{J\backslash\{\beta\}}|}[\Pet_{J\backslash\{\beta\}}]_S.
\end{align*}
\end{proof}

\begin{example}
\label{ex1}
Let $\Delta=B_2$.
\begin{center}
\begin{tikzpicture}
\draw[fill=blue] (0,0) circle (.1cm);
\draw[fill=blue] (1,0) circle (.1cm);
\draw (.1,.03)--+(0.8,0);
\draw (.1,-.03)--+(0.8,0);
\draw (0.6,0)--+(-0.12,0.12);
\draw (0.6,0)--+(-0.12,-0.12);
\draw (0,-.2) node {$\scriptscriptstyle 1$};
\draw (1,-.2) node {$\scriptscriptstyle 2$};
\end{tikzpicture}
\end{center}
We use \cref{thm:chevalley} to compute 
\begin{align*}
p_1\cap[\Pet]_S=
\left\langle2\rho^\vee,\varpi_1\right\rangle t\,[\Pet]_S+
\left\langle\varpi^{\vee}_1 ,\varpi_1\right\rangle
\frac{|W|}{|W_{\{2\}}|}[\Pet_{\{2\}}]_S+
\left\langle\varpi^{\vee}_2,\varpi_1\right\rangle
\frac{|W|}{|W_{\{1\}}|}[\Pet_{\{1\}}]_S.
\end{align*}
Recall the realization of the root system of $B_2$ in $\mathbb R^2$,
given by $\alpha_1=\epsilon_1-\epsilon_2$, $\alpha_2=\epsilon_2$.
We have $\varpi_1=\epsilon_1=\alpha_1+\alpha_2$.
Observe that $\left\langle \varpi_j^\vee,\varpi_i\right\rangle$ equals the coefficient of $\alpha_j$ in the expansion of $\varpi_i$ as a sum of simple roots.
Hence, we have
\begin{align*}
\left\langle\varpi_1^\vee,\varpi_1\right\rangle=1,&&
\left\langle\varpi_2^\vee,\varpi_1\right\rangle=1.
\end{align*}
Furthermore, $2\rho^\vee=2\varpi_1^\vee+2\varpi_2^\vee$,
and hence $\left\langle 2\rho^\vee,\varpi_1\right\rangle=4$.
Finally, we have $|W|=8$ and $|W_{\{1\}}|=|W_{\{2\}}|=2$.
Therefore,
\begin{align*}
p_1\cap[\Pet]_S=4t[\Pet]_S+4[\Pet_{\{1\}}]_S+4[\Pet_{\{2\}}]_S.
\end{align*}
\qed
\end{example}

Recall that $\Omega_I=\prod_{\alpha\in I}p_\alpha$,
so that, in particular, $\Omega_\alpha=p_\alpha$.
Following \cref{prop:duality,Giambelli}, the set $\set{\Omega_I}{I\subset\Delta}$ is a basis of $H^*_S(\Pet)$ over $\IQ[t]$. 
Using \cref{prop:duality,thm:mult},
we can deduce a Monk rule for this basis from the Chevalley formula.
\begin{thm}[Equivariant Monk Rule]
\label{thm:monk}
Let $f_J$ denote the \emph{connection index} of the Dynkin diagram $J$,
i.e., the determinant of the Cartan matrix of $J$.
For $\alpha\in\Delta$, we have
\begin{align*}
\Omega_\alpha \Omega_I=
\begin{cases}
\Omega_{I\cup\{\alpha\}}
&\text{if }\alpha\not\in I,\\
2\left\langle\rho^\vee_I,\varpi^I_\alpha\right\rangle t \Omega_I+
\sum\limits_{\gamma\in\Delta\backslash I}
\dfrac{f_{I\cup\{\gamma\}}}{f_I}\left\langle\varpi_\gamma^{(I\cup\{\gamma\})\vee},\varpi^{I\cup\{\gamma\}}_\alpha\right\rangle \Omega_{I\cup\{\gamma\}}
&\text{if }\alpha\in I.
\end{cases}
\end{align*}
\end{thm}
\begin{proof}
Consider 
the coefficients $c_{\alpha I}^J\in\IQ[t]$ in the product  $\Omega_\alpha \Omega_I=\sum c_{\alpha I}^J \Omega_J$.
Following \cref{prop:duality} we have
\begin{equation*}
c_{\alpha I}^J=
\frac{\left\langle \Omega_\alpha \Omega_I,[\Pet_J]_S\right\rangle}
{\left\langle \Omega_J,[\Pet_J]_S\right\rangle}
=\frac{\left\langle \Omega_I,\Omega_\alpha\cap [\Pet_J]_S\right\rangle}
{\left\langle \Omega_J,[\Pet_J]_S\right\rangle}
=\frac{f_J}{|W_J|}\left\langle \Omega_I,\Omega_\alpha\cap [\Pet_J]_S\right\rangle,
\end{equation*}
where the final equality is from \cref{lem:mult}.

Consider first $\alpha\not\in J$.
Following \cref{thm:chevalley}, we have $\Omega_\alpha\cap[\Pet_J]_S=0$,
and hence $c_{\alpha I}^J=0$.

Consider now $\alpha\in J$.
Recall from \cref{prop:duality} that $\left\langle\Omega_I,[\Pet_K]_S\right\rangle=0$ unless $I=K$.
Further, by \cref{thm:chevalley}, the only $[\Pet_I]_S$ appearing in the expansion of $\Omega_\alpha\cap[\Pet_J]_S$
correspond to  $I=J$ or $I=J\backslash\{\gamma\}$ for some $\gamma\in J$.
Thus $c_{\alpha I}^J=0$ unless $I=J$ or $I=J\backslash\{\gamma\}$ for some $\gamma\in J$. 
For $J=I$, we have,
\begin{align*}
c_{\alpha I}^I
&=\frac{ f_I}{|W_I|}\left\langle\Omega_I, \Omega_\alpha\cap [\Pet_I]_S\right\rangle\\
&=\frac{f_I}{|W_I|}\left\langle \Omega_I, \left\langle2\rho_I^\vee,\varpi^I_\alpha\right\rangle t\,[\Pet_I]_S
+ \sum\limits_{\beta\in I} \left\langle\varpi^{I\vee}_\beta,\varpi^I_\alpha\right\rangle\frac{|W_I|}{|W_{I\backslash\{\beta\}}|}[\Pet_{I\backslash\{\beta\}}]_S \right\rangle\\
&=\frac{f_I}{|W_I|}\left\langle2\rho_I^\vee,\varpi^I_\alpha\right\rangle t\,\left\langle \Omega_I,[\Pet_I]_S\right\rangle\\
&= \left\langle2\rho_I^\vee,\varpi^I_\alpha\right\rangle t.
\end{align*}
In the case where $J=I\sqcup\{\gamma\}$ for some $\gamma\in\Delta$,
we have
\begin{align*}
c_{\alpha I}^J&=
\frac{ f_J}{|W_J|}\left\langle \Omega_I,\Omega_\alpha\cap [\Pet_J]\right\rangle\\
&=\frac{f_J}{|W_J|}\left\langle \Omega_I, \left\langle2\rho_J^\vee,\varpi^J_\alpha\right\rangle t\,[\Pet_J]_S
+ \sum\limits_{\beta\in J} \left\langle\varpi^{J\vee}_\beta,\varpi^J_\alpha\right\rangle\frac{|W_J|}{|W_{J\backslash\{\beta\}}|}[\Pet_{J\backslash\{\beta\}}]_S\right\rangle
\\
&=\frac{f_J}{|W_J|}\frac{|W_J|}{|W_I|}
\left\langle\varpi^{J\vee}_\gamma,\varpi^J_\alpha\right\rangle
\left\langle \Omega_I,
[\Pet_I]_S
\right\rangle \\
&= \left\langle\varpi^{J\vee}_\gamma,\varpi^J_\alpha\right\rangle
 \frac{f_J}{f_I}.
\end{align*}
\end{proof}
\begin{example}
Consider $\Delta=B_3$, and let $I=\{1,2\}\subset\Delta$ .
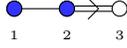
\begin{center}
\begin{tikzpicture}
\def\xl{-1.45}
\def\yl{0.15}
\def\xw{.7}
\def\n{3}

\foreach \x in {1,...,\n} 
{
 \draw (\xw*\x-\xw+\xl , \yl ) circle (.1cm);
 \draw (\xw*\x-1*\xw+\xl , \yl-.5*\xw ) node {$\scriptscriptstyle{\pgfmathparse{int(\x)}\pgfmathresult}$};
}
\foreach \x in {3,...,\n} \draw[xshift=3] (\xw*\x-3*\xw+\xl , \yl) --+(\xw-.2,0);

\draw[xshift=3,yshift=1] (\xw*\n-2*\xw+\xl , \yl) --+(\xw-.2,0);
\draw[xshift=3,yshift=-1] (\xw*\n-2*\xw+\xl , \yl) --+(\xw-.2,0);
\draw (\xw*\n-1.4*\xw+\xl , \yl) --+(-.2*\xw,.2*\xw);
\draw (\xw*\n-1.4*\xw+\xl , \yl) --+(-.2*\xw,-.2*\xw);
\draw[fill=blue!80] (\xl , \yl ) circle (.1cm);
\draw[fill=blue!80] (\xl+\xw , \yl ) circle (.1cm);
\end{tikzpicture}
\end{center}
We compute the product $\Omega_{2}\Omega_I$.
By \cref{thm:monk},
\begin{align*}
\Omega_{2}\Omega_I&=
2\left\langle\rho^\vee_I,\varpi_{2}^I\right\rangle t \Omega_I+
\sum\limits_{\gamma\in\Delta\backslash I}
\dfrac{f_{I\cup\{\gamma\}}}{f_I}\left\langle\varpi_\gamma^{(I\cup\{\gamma\})\vee},\varpi^{I\cup\{\gamma\}}_{2}\right\rangle \Omega_{I\cup\{\gamma\}}\\
&=
2\left\langle\rho^\vee_I,\varpi_{2}^I\right\rangle t \Omega_I+
\dfrac{f_{\Delta}}{f_I} \left\langle\varpi_{3}^{\vee},\varpi^{}_{2}\right\rangle \Omega_\Delta.
\end{align*}
Since the subdiagram $I$ is isomorphic to $A_2$,
the term $\left\langle\rho_I^\vee,\varpi_2^I\right\rangle$ is calculated in $A_2$.
We have $\rho_I^\vee=\frac12(\alpha_1^\vee+\alpha_2^\vee+(\alpha_1^\vee+\alpha_2^\vee))=\alpha_1^\vee+\alpha_2^\vee$,
and hence $\langle\rho_I^\vee,\varpi_2^I\rangle=1$.

The term $\left\langle\varpi_3^\vee,\varpi_2\right\rangle$ is the coefficient of $\alpha_3$ in the expansion of $\varpi_2$ as a sum of simple roots.
Recall the usual realization of the $B_3$ root system inside a $3$-dimensional vector space with orthonormal basis $\left\{\epsilon_1,\epsilon_2,\epsilon_3\right\}$,
given by $\alpha_1=\epsilon_1-\epsilon_2$, $\alpha_2=\epsilon_2-\epsilon_3$, and $\alpha_3=\epsilon_3$.
The fundamental weight $\varpi_2$ is given by
\begin{align*}
\varpi_2=\epsilon_1+\epsilon_2=\alpha_1+2\alpha_2+2\alpha_3,
\end{align*}
and hence $\left\langle\varpi_3^\vee,\varpi_2\right\rangle=2$.
The connection indices are
\begin{align*}
f_I=\det\begin{pmatrix}2&-1\\-1&2\end{pmatrix}=3,&&
f_\Delta=\det\begin{pmatrix}2&-1&0\\-1&2&-2\\0&-1&2\\\end{pmatrix}=2.
\end{align*}
Hence, we have $\Omega_2\Omega_I=2t\Omega_I+\frac43\Omega_\Delta$.
\end{example}

\begin{rem}
\label{rem:monk}
Fix a Coxeter element $v_I$ for each $I\subset\Delta$,
and set $p_{v_I}=i^*\sigma_{v_I}^S$.
It is common in the literature to work with the basis $\set{p_{v_I}}{I\subset\Delta}$.
We see from \cref{Giambelli} that $\set{p_{v_I}}{I\subset\Delta}$ and $\set{\Omega_I}{I\subset\Delta}$
are related by a diagonal change of basis matrix.
This allows us to translate \cref{thm:monk} into a Monk rule for the basis $\set{p_{v_I}}{I\subset\Delta}$,
\begin{align*}
p_\alpha p_{v_I}=
\begin{cases}
\dfrac{(|I|+1)R(v_I)}{R(v_{I\cup\{\alpha\}})}p_{v_{I\cup\{\alpha\}}}
&\text{if }\alpha\not\in I,\\
2 
\left\langle\rho^\vee_I,\varpi_\alpha\right\rangle t p_{v_I}+
\sum\limits_{\gamma\in\Delta\backslash I}
\dfrac{(|I|+1)f_{I\cup\{\gamma\}} R(v_I)}{f_I R(v_{I\cup\{\gamma\}})}
\left\langle\varpi_\gamma^{I\cup\{\gamma\}\vee},\varpi^{I\cup\{\gamma\}}_\alpha\right\rangle p_{v_{I\cup\{\gamma\}}}
&\text{if }\alpha\in I.
\end{cases}
\end{align*}
\end{rem}

\subsection{Tables of Structure Constants}
\label{sec:examples}
Observe that $\left\langle\varpi_\gamma^{J\vee},\varpi_\alpha^J\right\rangle$ is precisely the coefficient of $\gamma$ in the expression of the fundamental weight $\varpi_\alpha^J$ as a sum of the simple roots in $J$.
These coefficients, and the connection indices of the Dynkin diagrams, are listed in \cite[Tables 2, 3]{onishchik.vinberg},
see also \cite[Ch.6, \S4]{bourbaki:Lie46}.
Using this, we can compute the structure constants in the Chevalley and Monk formulae.
For the reader's convenience, we tabulate the equivariant structure constants for the Monk and Chevalley formulae in types A--D.
The ordinary structure constants in the Monk rule for all Dynkin diagrams have also been recently computed and tabulated by Horiguchi in \cite[Table 2]{horiguchi2021mixed}.
\footnote{Our structure constants $c_{i,K}^J$ correspond to the numbers $m_{i,K}^J$ in \cite{horiguchi2021mixed}.}


We will denote by $c_{iJ}^K$ and $d_{iJ}^K$ the structure constants given by
\begin{align*}
\Omega_{\alpha_i}\Omega_I=\sum c_{iI}^J\Omega_J,&&
p_{\alpha_i}\cap[\Pet_J]_S=\sum d_{iJ}^K[\Pet_K]_S
\end{align*}
respectively. We write $i$ for $\alpha_i\in I$ to simplify notation.
Following \cref{thm:chevalley,thm:monk}, for $i\in I$, we have
\begin{align*}
&&c_{iI}^I & =d_{iI}^I =\left\langle2\rho_I^\vee,\varpi^I_i\right\rangle t,         & \\
&&c_{iI}^J & = \dfrac{f_J}{f_I}\left\langle\varpi_j^{J\vee},\varpi^J_i\right\rangle & J=I\sqcup\{j\},
\end{align*}
and for $i\in J$, we have
\begin{align*}
&&d_{iJ}^K=\left\langle\varpi^{J\vee}_j,\varpi^J_i\right\rangle \frac{|W_J|}{|W_K|},&&  K=J\backslash\{j\}.
\end{align*}
For the classical type Dynkin diagrams, these values are listed in \cref{MonkTable}.
The ordinary structure constants for the Chevalley formula are listed in \cref{TableChevalley}.
The structure constants not involving $t$, in the Monk rule, are listed (for types A--D) in \cref{ordCM}.

\begin{table}[ht]
\begin{tabular}{| c | c |}
\hline
\begin{tikzpicture}\draw(0,0)node{$I$};\end{tikzpicture} &\begin{tikzpicture}\draw(0,0)node{$c_{iI}^I=d_{iI}^I
$};\end{tikzpicture}\\
\hline
{
\def\k{4}
\def\d{0}
\begin{tikzpicture}
\def\xl{0}
\def\yl{.3}
\def\xw{.7}

\draw (\xw+\xl , \yl+0.4 ) node{};
\foreach \x in {1,...,\k} 
{
 \draw (\xw*\x-\xw+\xl , \yl ) circle (.1cm);
 \draw[xshift=3] (\xw*\x-\xw+\xl , \yl) --+(\xw-.2,0);
}
 \draw (\xl , \yl-.5*\xw ) node {$\scriptscriptstyle{1}$};
 \draw (\xl+\xw , \yl-.5*\xw ) node {$\scriptscriptstyle{2}$};
 \draw (\xl+4*\xw , \yl-.5*\xw ) node {$\scriptscriptstyle{n}$};

\foreach \x in {2,...,\d} 
{
 \draw (\xw*\x-2*\xw+\k*\xw+\xl , \yl ) circle (.1cm);
 \draw[xshift=3] (\xw*\x+\k*\xw-3*\xw+\xl , \yl) --+(\xw-.2,0);
}
\end{tikzpicture}
} &
\begin{tikzpicture}\draw(0,0)node{$ i(n+1-i)t$};\draw(0,-0.33) node{};\end{tikzpicture}\\
\hline
\begin{tikzpicture}
\def\xl{-4.2}
\def\yl{1.3}
\def\xw{.7}
\def\n{5}

\draw(-2,0.7) node{};
\draw (\xw+\xl , \yl+0.4 ) node{};
\foreach \x in {2,...,\n}
{
 \draw (\xw*\x-2*\xw+\xl , \yl ) circle (.1cm);
}
\draw (\xl , \yl-.5*\xw ) node {$\scriptscriptstyle 1$};
\draw (\xw+\xl , \yl-.5*\xw ) node {$\scriptscriptstyle 2$};
\draw (4*\xw+\xl , \yl-.5*\xw ) node {$\scriptscriptstyle{n}$};
\foreach \x in {3,...,\n} \draw[xshift=3] (\xw*\x-3*\xw+\xl , \yl) --+(\xw-.2,0);

\draw[xshift=3,yshift=1] (\xw*\n-2*\xw+\xl , \yl) --+(\xw-.2,0);
\draw[xshift=3,yshift=-1] (\xw*\n-2*\xw+\xl , \yl) --+(\xw-.2,0);
\draw (\xw*\n-1.35*\xw+\xl , \yl) --+(-.2*\xw,.2*\xw);
\draw (\xw*\n-1.35*\xw+\xl , \yl) --+(-.2*\xw,-.2*\xw);
\draw (\xl+\n*\xw-\xw , \yl ) circle (.1cm);
\end{tikzpicture} &
\begin{tikzpicture}\draw(0,0)node{$ \begin{cases}\dfrac {n(n+1)}2t& \text{if }i=n\\ i(2n+1-i)t & \text{if }i\neq n\end{cases}$};\end{tikzpicture}\\
\hline
\begin{tikzpicture}
\def\xl{-4.2}
\def\yl{1.3}
\def\xw{.7}
\def\n{5}

\draw (\xw+\xl , \yl+0.4 ) node{};
\foreach \x in {2,...,\n}
{
\draw (\xw*\x-2*\xw+\xl , \yl ) circle (.1cm);
}
\draw (\xl , \yl-.5*\xw ) node {$\scriptscriptstyle 1$};
\draw (\xw+\xl , \yl-.5*\xw ) node {$\scriptscriptstyle 2$};
\draw (4*\xw+\xl , \yl-.5*\xw ) node {$\scriptscriptstyle{n}$};

\foreach \x in {3,...,\n} \draw[xshift=3] (\xw*\x-3*\xw+\xl , \yl) --+(\xw-.2,0);

\draw[xshift=3,yshift=1] (\xw*\n-2*\xw+\xl , \yl) --+(\xw-.2,0);
\draw[xshift=3,yshift=-1] (\xw*\n-2*\xw+\xl , \yl) --+(\xw-.2,0);
\draw (\xw*\n-1.55*\xw+\xl , \yl) --+(.2*\xw,.2*\xw);
\draw (\xw*\n-1.55*\xw+\xl , \yl) --+(.2*\xw,-.2*\xw);
\draw (\xl+\n*\xw-\xw , \yl ) circle (.1cm);
\end{tikzpicture} &
\begin{tikzpicture}\draw(0,0)node{$i(2n-i)t$};\draw(0,-0.3)node{};\end{tikzpicture}\\
\hline
{\def\y{1.4}
\begin{tikzpicture}
\def\n{6}
\def\xl{-4.3}
\def\yl{0.6}
\def\xw{.7}

\draw (-3,\y) node {};

\foreach \x in {4,...,\n} 
{
\draw (\xw*\x-3*\xw+\xl , \yl ) circle (.1cm);
\draw[xshift=3] (\xw*\x-4*\xw+\xl , \yl) --+(\xw-.2,0);
}
 \draw (\xw*4-3*\xw+\xl , \yl-.5*\xw ) node {$\scriptscriptstyle{\pgfmathparse{int(4-2)}\pgfmathresult}$};
\draw (\xl , \yl ) circle (.1cm);
\draw (\xl , \yl-.5*\xw ) node {$\scriptscriptstyle1$};
\draw (\xl+\n*\xw-2*\xw , \yl+\xw ) node {$\scriptscriptstyle{n-1}$};
\draw (\xl+\n*\xw-2*\xw , \yl-\xw ) node {$\scriptscriptstyle{n}$};
\draw[yshift=2,xshift=2] (\xl+\n*\xw-3*\xw , \yl) --+(\xw*.53,\xw*.53);
\draw[yshift=-2,xshift=2] (\xl+\n*\xw-3*\xw , \yl) --+(\xw*.53,-\xw*.53);
\draw(\xl+\n*\xw-2.27*\xw , \yl+.73*\xw) circle (.1cm);
\draw (\xl+\n*\xw-2.27*\xw , \yl-.73*\xw) circle (.1cm);
\end{tikzpicture}} &
\begin{tikzpicture}\draw(0,0)node{$\begin{cases}\dfrac{n(n-1)}2t & \text{if }i\in\{n-1,n\}\\i(2n-1-i)t& \text{otherwise.}\end{cases}$};\draw(0,-0.7) node{};\end{tikzpicture}\\
\hline 
\end{tabular}
\caption{Structure constant for the leading term in the Monk and Chevalley formulae.} 
\label{MonkTable}
\end{table}

\begin{table}
\begin{tabular}{| c | c |}
\hline
\begin{tikzpicture}\draw(0,0)node{$I$};\end{tikzpicture} &\begin{tikzpicture}\draw(0,0)node{$d_{iI}^{I\backslash\{j\}}$};\end{tikzpicture}\\
\hline
{
\def\k{4}
\def\d{0}
\begin{tikzpicture}
\def\xl{0}
\def\yl{.3}
\def\xw{.7}

\draw (\xw+\xl , \yl+0.4 ) node{};
\foreach \x in {1,...,\k} 
{
 \draw (\xw*\x-\xw+\xl , \yl ) circle (.1cm);
 \draw[xshift=3] (\xw*\x-\xw+\xl , \yl) --+(\xw-.2,0);
}
 \draw (\xl , \yl-.5*\xw ) node {$\scriptscriptstyle{1}$};
 \draw (\xl+\xw , \yl-.5*\xw ) node {$\scriptscriptstyle{2}$};
 \draw (\xl+4*\xw , \yl-.5*\xw ) node {$\scriptscriptstyle{n}$};

\foreach \x in {2,...,\d} 
{
 \draw (\xw*\x-2*\xw+\k*\xw+\xl , \yl ) circle (.1cm);
 \draw[xshift=3] (\xw*\x+\k*\xw-3*\xw+\xl , \yl) --+(\xw-.2,0);
}
\end{tikzpicture}
} &
\begin{tikzpicture}
\draw(0,0)node{$\begin{cases}\binom{n}{j-1}{(n-i+1)}&\text{if }j\leq i,\\\binom{n}ji&\text{if }j\geq i.\end{cases}$};
\end{tikzpicture}\\
\hline
\begin{tikzpicture}
\def\xl{-4.2}
\def\yl{1.3}
\def\xw{.7}
\def\n{5}

\draw(-2,0.7) node{};
\draw (\xw+\xl , \yl+0.4 ) node{};
\foreach \x in {2,...,\n}
{
 \draw (\xw*\x-2*\xw+\xl , \yl ) circle (.1cm);
}
\draw (\xl , \yl-.5*\xw ) node {$\scriptscriptstyle 1$};
\draw (\xw+\xl , \yl-.5*\xw ) node {$\scriptscriptstyle 2$};
\draw (4*\xw+\xl , \yl-.5*\xw ) node {$\scriptscriptstyle{n}$};
\foreach \x in {3,...,\n} \draw[xshift=3] (\xw*\x-3*\xw+\xl , \yl) --+(\xw-.2,0);

\draw[xshift=3,yshift=1] (\xw*\n-2*\xw+\xl , \yl) --+(\xw-.2,0);
\draw[xshift=3,yshift=-1] (\xw*\n-2*\xw+\xl , \yl) --+(\xw-.2,0);
\draw (\xw*\n-1.35*\xw+\xl , \yl) --+(-.2*\xw,.2*\xw);
\draw (\xw*\n-1.35*\xw+\xl , \yl) --+(-.2*\xw,-.2*\xw);
\draw (\xl+\n*\xw-\xw , \yl ) circle (.1cm);
\end{tikzpicture} &
\begin{tikzpicture}
\draw(0,0)node{$\begin{cases}2^{j-1}j\binom nj&\text{if }i=n,\\2^j\binom nj\min(i,j)&\text{otherwise.}\end{cases}$};
\end{tikzpicture}\\
\hline
\begin{tikzpicture}
\def\xl{-4.2}
\def\yl{1.3}
\def\xw{.7}
\def\n{5}

\draw (\xw+\xl , \yl+0.4 ) node{};
\foreach \x in {2,...,\n}
{
\draw (\xw*\x-2*\xw+\xl , \yl ) circle (.1cm);
}
\draw (\xl , \yl-.5*\xw ) node {$\scriptscriptstyle 1$};
\draw (\xw+\xl , \yl-.5*\xw ) node {$\scriptscriptstyle 2$};
\draw (4*\xw+\xl , \yl-.5*\xw ) node {$\scriptscriptstyle{n}$};

\foreach \x in {3,...,\n} \draw[xshift=3] (\xw*\x-3*\xw+\xl , \yl) --+(\xw-.2,0);

\draw[xshift=3,yshift=1] (\xw*\n-2*\xw+\xl , \yl) --+(\xw-.2,0);
\draw[xshift=3,yshift=-1] (\xw*\n-2*\xw+\xl , \yl) --+(\xw-.2,0);
\draw (\xw*\n-1.55*\xw+\xl , \yl) --+(.2*\xw,.2*\xw);
\draw (\xw*\n-1.55*\xw+\xl , \yl) --+(.2*\xw,-.2*\xw);
\draw (\xl+\n*\xw-\xw , \yl ) circle (.1cm);
\end{tikzpicture} &
\begin{tikzpicture}
\draw(0,0)node{$\begin{cases}2^{n-1}i&\text{if }j=n,\\2^j\binom nj\min(i,j)&\text{otherwise.}\end{cases}$};
\end{tikzpicture}\\
\hline
{\def\y{-0.6}
\begin{tikzpicture}
\def\n{6}
\def\xl{-4.3}
\def\yl{0.6}
\def\xw{.7}

\draw (-3,\y) node {};

\foreach \x in {4,...,\n} 
{
\draw (\xw*\x-3*\xw+\xl , \yl ) circle (.1cm);
\draw[xshift=3] (\xw*\x-4*\xw+\xl , \yl) --+(\xw-.2,0);
}
 \draw (\xw*4-3*\xw+\xl , \yl-.5*\xw ) node {$\scriptscriptstyle{\pgfmathparse{int(4-2)}\pgfmathresult}$};
\draw (\xl , \yl ) circle (.1cm);
\draw (\xl , \yl-.5*\xw ) node {$\scriptscriptstyle1$};
\draw (\xl+\n*\xw-2*\xw , \yl+\xw ) node {$\scriptscriptstyle{n-1}$};
\draw (\xl+\n*\xw-2*\xw , \yl-\xw ) node {$\scriptscriptstyle{n}$};
\draw[yshift=2,xshift=2] (\xl+\n*\xw-3*\xw , \yl) --+(\xw*.53,\xw*.53);
\draw[yshift=-2,xshift=2] (\xl+\n*\xw-3*\xw , \yl) --+(\xw*.53,-\xw*.53);
\draw(\xl+\n*\xw-2.27*\xw , \yl+.73*\xw) circle (.1cm);
\draw (\xl+\n*\xw-2.27*\xw , \yl-.73*\xw) circle (.1cm);
\end{tikzpicture}} & 
\begin{tikzpicture}\draw(0,0)node{$
\begin{cases}
2^{n-3}n&\text{if }i=j\geq n-1,\\
2^{n-3}(n-2)&\text{if }\{i,j\}=\{n-1,n\},\\ 
2^{n-2}i&\text{if }i\leq n-2\text{ and }j\geq n-1,\\
2^{j-1}j\binom nj&\text{if }i\geq n-1\text{ and }j\leq n-2\\ 
2^j\binom nj\min(i,j)&\text{if }i,j\leq n-2.
\end{cases}
$};
\end{tikzpicture}\\
\hline 
\end{tabular}
\caption{Structure constants for the Chevalley formula.}
\label{TableChevalley}
\end{table}
\begin{table}[ht]
\centering
\begin{tabular}{| c | c |}
\hline
Dynkin Pair $\begin{tikzcd}({\color{blue}J},K)\end{tikzcd}$& $c_{i J}^K$ \\ \hline
{
\def\k{4}
\def\d{1}
\begin{tikzpicture}
\def\xl{0}
\def\yl{.3}
\def\xw{.7}

\draw (\xw+\xl , \yl+0.4 ) node{};
\foreach \x in {1,...,\k} 
{
 \draw[fill=blue!80] (\xw*\x-\xw+\xl , \yl ) circle (.1cm);
 \draw[xshift=3] (\xw*\x-\xw+\xl , \yl) --+(\xw-.2,0);
}
 \draw (\xl , \yl-.5*\xw ) node {$\scriptscriptstyle{1}$};
 \draw (\xl+4*\xw , \yl-.5*\xw ) node {$\scriptscriptstyle{n}$};

\foreach \x in {2,...,\d} 
{
 \draw (\xw*\x-2*\xw+\k*\xw+\xl , \yl ) circle (.1cm);
 \draw[xshift=3] (\xw*\x+\k*\xw-3*\xw+\xl , \yl) --+(\xw-.2,0);
}
\end{tikzpicture}
} & \begin{tikzpicture}\draw(0,0)node {$\dfrac in$};\end{tikzpicture} \\\hline
\begin{tikzpicture}
\def\xl{-4.2}
\def\yl{1.3}
\def\xw{.7}
\def\n{5}

\draw (\xw+\xl , \yl+0.4 ) node{};
\foreach \x in {2,...,\n}
{
 \draw[fill=blue!80] (\xw*\x-2*\xw+\xl , \yl ) circle (.1cm);
}
\draw (\xl , \yl-.5*\xw ) node {$\scriptscriptstyle 1$};
\draw (\xw+\xl , \yl-.5*\xw ) node {$\scriptscriptstyle 2$};
\draw (4*\xw+\xl , \yl-.5*\xw ) node {$\scriptscriptstyle{n}$};
\foreach \x in {3,...,\n} \draw[xshift=3] (\xw*\x-3*\xw+\xl , \yl) --+(\xw-.2,0);

\draw[xshift=3,yshift=1] (\xw*\n-2*\xw+\xl , \yl) --+(\xw-.2,0);
\draw[xshift=3,yshift=-1] (\xw*\n-2*\xw+\xl , \yl) --+(\xw-.2,0);
\draw (\xw*\n-1.35*\xw+\xl , \yl) --+(-.2*\xw,.2*\xw);
\draw (\xw*\n-1.35*\xw+\xl , \yl) --+(-.2*\xw,-.2*\xw);
\draw (\xl+\n*\xw-\xw , \yl ) circle (.1cm);
\end{tikzpicture} & \begin{tikzpicture}\draw(0,0) node {$\dfrac{2i}n$};\end{tikzpicture} \\ \hline
\begin{tikzpicture}
\def\xl{-4.2}
\def\yl{1.3}
\def\xw{.7}
\def\n{5}

\foreach \x in {3,...,\n}
{
 \draw[fill=blue!80] (\xw*\x-2*\xw+\xl , \yl ) circle (.1cm);
}
\draw (\xl , \yl-.5*\xw ) node {$\scriptscriptstyle 1$};
\draw (\xw+\xl , \yl-.5*\xw ) node {$\scriptscriptstyle 2$};
\draw (4*\xw+\xl , \yl-.5*\xw ) node {$\scriptscriptstyle{n}$};

\draw (\xl , \yl ) circle (.1cm);
\foreach \x in {3,...,\n} \draw[xshift=3] (\xw*\x-3*\xw+\xl , \yl) --+(\xw-.2,0);

\draw[xshift=3,yshift=1] (\xw*\n-2*\xw+\xl , \yl) --+(\xw-.2,0);
\draw[xshift=3,yshift=-1] (\xw*\n-2*\xw+\xl , \yl) --+(\xw-.2,0);
\draw (\xw*\n-1.35*\xw+\xl , \yl) --+(-.2*\xw,.2*\xw);
\draw (\xw*\n-1.35*\xw+\xl , \yl) --+(-.2*\xw,-.2*\xw);
\draw[fill=blue!80] (\xl+\n*\xw-\xw , \yl ) circle (.1cm);
\end{tikzpicture} & \begin{tikzpicture}\draw(0,0) node{$\begin{cases}1&\text{if }i\neq n\\\frac12&\text{otherwise.}\end{cases}$};\end{tikzpicture} \\ \hline
\begin{tikzpicture}
\def\xl{-4.2}
\def\yl{1.3}
\def\xw{.7}
\def\n{5}

\draw (\xw+\xl , \yl+0.4 ) node{};
\foreach \x in {2,...,\n}
{
\draw[fill=blue!80] (\xw*\x-2*\xw+\xl , \yl ) circle (.1cm);
}
\draw (\xl , \yl-.5*\xw ) node {$\scriptscriptstyle 1$};
\draw (\xw+\xl , \yl-.5*\xw ) node {$\scriptscriptstyle 2$};
\draw (4*\xw+\xl , \yl-.5*\xw ) node {$\scriptscriptstyle{n}$};

\foreach \x in {3,...,\n} \draw[xshift=3] (\xw*\x-3*\xw+\xl , \yl) --+(\xw-.2,0);

\draw[xshift=3,yshift=1] (\xw*\n-2*\xw+\xl , \yl) --+(\xw-.2,0);
\draw[xshift=3,yshift=-1] (\xw*\n-2*\xw+\xl , \yl) --+(\xw-.2,0);
\draw (\xw*\n-1.55*\xw+\xl , \yl) --+(.2*\xw,.2*\xw);
\draw (\xw*\n-1.55*\xw+\xl , \yl) --+(.2*\xw,-.2*\xw);
\draw (\xl+\n*\xw-\xw , \yl ) circle (.1cm);
\end{tikzpicture} & \begin{tikzpicture}\draw(0,0) node{$\dfrac in$};\end{tikzpicture} \\ \hline
\begin{tikzpicture}
\def\xl{-4.2}
\def\yl{1.3}
\def\xw{.7}
\def\n{5}

\draw (\xw+\xl , \yl+0.4 ) node{};
\foreach \x in {3,...,\n}
{
 \draw[fill=blue!80] (\xw*\x-2*\xw+\xl , \yl ) circle (.1cm);
}
\draw (\xl , \yl-.5*\xw ) node {$\scriptscriptstyle 1$};
\draw (\xw+\xl , \yl-.5*\xw ) node {$\scriptscriptstyle 2$};
\draw (4*\xw+\xl , \yl-.5*\xw ) node {$\scriptscriptstyle{n}$};

\draw (\xl , \yl ) circle (.1cm);
\foreach \x in {3,...,\n} \draw[xshift=3] (\xw*\x-3*\xw+\xl , \yl) --+(\xw-.2,0);

\draw[xshift=3,yshift=1] (\xw*\n-2*\xw+\xl , \yl) --+(\xw-.2,0);
\draw[xshift=3,yshift=-1] (\xw*\n-2*\xw+\xl , \yl) --+(\xw-.2,0);
\draw (\xw*\n-1.55*\xw+\xl , \yl) --+(.2*\xw,.2*\xw);
\draw (\xw*\n-1.55*\xw+\xl , \yl) --+(.2*\xw,-.2*\xw);
\draw[fill=blue!80] (\xl+\n*\xw-\xw , \yl ) circle (.1cm);
\end{tikzpicture} & \begin{tikzpicture}\draw(0,0) node{$1$};\draw(0,-0.3)node{};\end{tikzpicture} \\ \hline
\begin{tikzpicture}
\def\n{6}
\def\xl{-4.3}
\def\yl{0.6}
\def\xw{.7}

\foreach \x in {4,...,\n} 
{
\draw[fill=blue!80] (\xw*\x-3*\xw+\xl , \yl ) circle (.1cm);
\draw[xshift=3] (\xw*\x-4*\xw+\xl , \yl) --+(\xw-.2,0);
}
 \draw (\xw*4-3*\xw+\xl , \yl-.5*\xw ) node {$\scriptscriptstyle{\pgfmathparse{int(4-2)}\pgfmathresult}$};
\draw[fill=blue!80] (\xl , \yl ) circle (.1cm);
\draw (\xl , \yl-.5*\xw ) node {$\scriptscriptstyle1$};
\draw (\xl+\n*\xw-2*\xw , \yl+\xw ) node {$\scriptscriptstyle{n-1}$};
\draw (\xl+\n*\xw-2*\xw , \yl-\xw ) node {$\scriptscriptstyle{n}$};
\draw[yshift=2,xshift=2] (\xl+\n*\xw-3*\xw , \yl) --+(\xw*.53,\xw*.53);
\draw[yshift=-2,xshift=2] (\xl+\n*\xw-3*\xw , \yl) --+(\xw*.53,-\xw*.53);
\draw[fill=blue!80](\xl+\n*\xw-2.27*\xw , \yl+.73*\xw) circle (.1cm);
\draw (\xl+\n*\xw-2.27*\xw , \yl-.73*\xw) circle (.1cm);
\end{tikzpicture} & \begin{tikzpicture}\draw(0,0) node{$\begin{cases}\dfrac{2i}n&\text{if }i\neq n-1,\\\dfrac{n-2}n&\text{otherwise.}\end{cases}$};\end{tikzpicture} \\ \hline
\begin{tikzpicture}
\def\n{6}
\def\xl{-4.3}
\def\yl{0.6}
\def\xw{.7}

\foreach \x in {4,...,\n} 
{
\draw[fill=blue!80] (\xw*\x-3*\xw+\xl , \yl ) circle (.1cm);
\draw[xshift=3] (\xw*\x-4*\xw+\xl , \yl) --+(\xw-.2,0);
}
 \draw (\xw*4-3*\xw+\xl , \yl-.5*\xw ) node {$\scriptscriptstyle{\pgfmathparse{int(4-2)}\pgfmathresult}$};
\draw (\xl , \yl ) circle (.1cm);
\draw (\xl , \yl-.5*\xw ) node {$\scriptscriptstyle1$};
\draw (\xl+\n*\xw-2*\xw , \yl+\xw ) node {$\scriptscriptstyle{n-1}$};
\draw (\xl+\n*\xw-2*\xw , \yl-\xw ) node {$\scriptscriptstyle{n}$};
\draw[yshift=2,xshift=2] (\xl+\n*\xw-3*\xw , \yl) --+(\xw*.53,\xw*.53);
\draw[yshift=-2,xshift=2] (\xl+\n*\xw-3*\xw , \yl) --+(\xw*.53,-\xw*.53);
\draw[fill=blue!80](\xl+\n*\xw-2.27*\xw , \yl+.73*\xw) circle (.1cm);
\draw[fill=blue!80] (\xl+\n*\xw-2.27*\xw , \yl-.73*\xw) circle (.1cm);
\end{tikzpicture} & \begin{tikzpicture}\draw(0,0) node{$\begin{cases}\frac12&\text{if }i\in\{n-1,n\},\\ 1&\text{otherwise.}\end{cases}$};\draw(0,-0.7)node{};\end{tikzpicture} \\ \hline
\end{tabular}
\caption{Ordinary terms in the Monk rule, see \cref{thm:monk}.}
\label{ordCM}
\end{table}

\begin{example}
\label{ex2}
Consider $\Delta=B_3$, and $I=\{2,3\}$.
We compute the product $p_2\cap[\Pet_I]_S$ using \cref{MonkTable,TableChevalley}.
\begin{center}
\begin{tikzpicture}
\def\xl{-1.45}
\def\yl{0.15}
\def\xw{.7}
\def\n{3}

\foreach \x in {2,...,\n} 
{
 \draw[fill=blue] (\xw*\x-\xw+\xl , \yl ) circle (.1cm);
 \draw (\xw*\x-1*\xw+\xl , \yl-.5*\xw ) node {$\scriptscriptstyle{\pgfmathparse{int(\x)}\pgfmathresult}$};
}
\draw (\xl , \yl-.5*\xw ) node {$\scriptscriptstyle{\pgfmathparse{int(1)}\pgfmathresult}$};
\foreach \x in {3,...,\n} \draw[xshift=3] (\xw*\x-3*\xw+\xl , \yl) --+(\xw-.2,0);
\draw[xshift=3,yshift=1] (\xw*\n-2*\xw+\xl , \yl) --+(\xw-.2,0);
\draw[xshift=3,yshift=-1] (\xw*\n-2*\xw+\xl , \yl) --+(\xw-.2,0);
\draw (\xw*\n-1.4*\xw+\xl , \yl) --+(-.2*\xw,.2*\xw);
\draw (\xw*\n-1.4*\xw+\xl , \yl) --+(-.2*\xw,-.2*\xw);
\draw (\xl , \yl ) circle (.1cm);
\end{tikzpicture}
\end{center}

We have
$
p_2\cap[\Pet_I]_S=d_{2I}^I[\Pet_I]
+d_{2I}^{\{2\}}[\Pet_{\{2\}}]_S+d_{2I}^{\{3\}}[\Pet_{\{3\}}]_S.
$
The subdiagram $I$ is isomorphic to $B_2$,
so the coefficient $d_{2I}^I=4t$ corresponds to $i=1$, $n=2$ in the second row of \cref{MonkTable}.
The coefficient $d_{2I}^{\{2\}}=4$ corresponds to $i=1$, $j=2$, and $n=2$ in the second row of \cref{TableChevalley}.
The coefficient $d_{2I}^{\{3\}}=4$ corresponds to $i=1$, $j=1$, and $n=2$ in the second row of \cref{TableChevalley}.
Hence, we have 
\begin{align*}
p_2\cap[\Pet_I]_S=4t[\Pet_I]_S+4[\Pet_{\{2\}}]_S+4[\Pet_{\{3\}}]_S.
\end{align*}
\qed
\end{example}
Recall from \cref{thm:chevalley} that the coefficients in the Chevalley formula do not depend on the superset $\Delta$ containing $I$.
This phenomenon can be observed by comparing \cref{ex1} with \cref{ex2}.

\begin{example}
Consider $\Delta=D_6$, and let $I=\{3,4,5\}\subset\Delta$ .
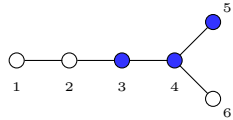
\begin{center}
\begin{tikzpicture}
\def\xl{-4.3}
\def\yl{0.6}
\def\xw{.7}
\def\n{6}

\draw (\xw*4-3*\xw+\xl , \yl ) circle (.1cm);
\draw[xshift=3] (\xw*4-4*\xw+\xl , \yl) --+(\xw-.2,0);
\draw (\xw*4-3*\xw+\xl , \yl-.5*\xw ) node {$\scriptscriptstyle 2$};

\foreach \x in {5,...,\n} 
{
 \draw[fill=blue!80] (\xw*\x-3*\xw+\xl , \yl ) circle (.1cm);
 \draw[xshift=3] (\xw*\x-4*\xw+\xl , \yl) --+(\xw-.2,0);
 \draw (\xw*\x-3*\xw+\xl , \yl-.5*\xw ) node {$\scriptscriptstyle{\pgfmathparse{int(\x-2)}\pgfmathresult}$};
}
\draw (\xl , \yl ) circle (.1cm);
\draw (\xl , \yl-.5*\xw ) node {$\scriptscriptstyle1$};
\draw (\xl+\n*\xw-2*\xw , \yl+\xw ) node {$\scriptscriptstyle{\pgfmathparse{int(\n-1)}\pgfmathresult}$};
\draw (\xl+\n*\xw-2*\xw , \yl-\xw ) node {$\scriptscriptstyle{\pgfmathparse{int(\n)}\pgfmathresult}$};
\draw[yshift=2,xshift=2] (\xl+\n*\xw-3*\xw , \yl) --+(\xw*.53,\xw*.53);
\draw[yshift=-2,xshift=2] (\xl+\n*\xw-3*\xw , \yl) --+(\xw*.53,-\xw*.53);
\draw (\xl+\n*\xw-2.27*\xw , \yl-.73*\xw) circle (.1cm);
\draw[fill=blue!80] (\xl+\n*\xw-2.27*\xw , \yl+.73*\xw) circle (.1cm);
\end{tikzpicture}
\end{center}
We compute the product $\Omega_{3}\Omega_I$ using \cref{MonkTable,ordCM}.
Observe first that for $\gamma\in\Delta\backslash I$, $J=I\cup\{\gamma\}$,
we have $\left\langle\varpi_\gamma^{J\vee},\varpi_\alpha^J\right\rangle=0$ if $\gamma$ is not connected to $I$.
We deduce that $c_{\alpha_3I}^{I\cup\{1\}}=0$, and that
\begin{align*}
\Omega_{3}\Omega_I
&=c_{3I}^I\Omega_I+c_{3I}^{I\cup\{2\}}\Omega_{I\cup\{2\}}+c_{3I}^{I\cup\{6\}}\Omega_{I\cup\{6\}}.
\end{align*}
The coefficient $c_{3I}^I=3t$ corresponds to $i=1$, $n=3$ in the first row of \cref{MonkTable},
The coefficient $c_{3I}^{I\cup\{6\}}=\frac12$ corresponds to 
$i=1$, $n=4$ in the sixth row of \cref{ordCM},
and the coefficient 
$
c_{3I}^{I\cup\{2\}}=\frac34$ corresponds to $i=3$ and $n=4$ in the first row of \cref{ordCM}.
Hence, 
\begin{align*}
\Omega_{3}\Omega_I=3t\Omega_I+\frac34\Omega_{I\cup\{2\}}+\frac12\Omega_{I\cup\{6\}}.
\end{align*}
\end{example}

\bibliographystyle{abbrv}
\bibliography{biblio}

\end{document}